\documentclass[10pt, reqno]{amsart}
\usepackage{silence}
\usepackage{xcolor}
\usepackage[T1]{fontenc}
\usepackage{lmodern}
\usepackage{amsxtra}
\usepackage{mathrsfs}
\usepackage{nicefrac,mathtools}

\usepackage[utf8]{inputenc}
\usepackage{amssymb,amscd,amsmath}
\usepackage{amsthm}
\usepackage{amsfonts}
\usepackage{enumerate}
\usepackage{multicol}
\usepackage[bookmarks]{hyperref}
\usepackage{bbm}
\usepackage{color}
\usepackage{xypic}
\usepackage{xspace}
\usepackage{todonotes}
\usepackage{verbatim}

\usepackage[lite]{amsrefs}

\renewcommand{\PrintDOI}[1]{\href{http://dx.doi.org/\detokenize{#1}}{doi: \detokenize{#1}}}

\BibSpec{book}{%
  +{}  {\PrintPrimary}                {transition}
  +{,} { \textit}                     {title}
  +{.} { }                            {part}
  +{:} { \textit}                     {subtitle}
  +{,} { \PrintEdition}               {edition}
  +{}  { \PrintEditorsB}              {editor}
  +{,} { \PrintTranslatorsC}          {translator}
  +{,} { \PrintContributions}         {contribution}
  +{,} { }                            {series}
  +{,} { \voltext}                    {volume}
  +{,} { }                            {publisher}
  +{,} { }                            {organization}
  +{,} { }                            {address}
  +{,} { \PrintDateB}                 {date}
  +{,} { }                            {status}
  +{}  { \parenthesize}               {language}
  +{}  { \PrintTranslation}           {translation}
  +{;} { \PrintReprint}               {reprint}
  +{.} { }                            {note}
  +{.} {}                             {transition}
  +{} { \PrintDOI}                   {doi}
  +{} { available at \url}            {eprint}
  +{}  {\SentenceSpace \PrintReviews} {review}
}

\DeclareMathOperator{\Prim}{Prim}
\DeclareMathOperator{\Ind}{Ind}

\newcommand{\al}{\alpha}
\newcommand{\be}{\beta}

\newcommand{\kA}{{\mathbbm{k}(\Acal)}}
\newcommand{\kB}{{\mathbbm{k}(\Bcal)}}
\newcommand{\kcA}{\mathbbm{k}_c(\Acal)}
\newcommand{\kcB}{\mathbbm{k}_c(\Bcal)}
\newcommand{\la}{\langle}
\newcommand{\laa}{{}_\Acal\langle}
\newcommand{\lab}{\langle}
\newcommand{\labe}{{}_{\Bcal_\be}\langle}
\newcommand{\lac}{\langle}
\newcommand{\lak}{{}_{\mathbbm{k}}\langle}
\newcommand{\LtwoB}{\mathcal{L}^2\mathcal{B}}

\newcommand{\id}{\operatorname{id}}
\newcommand{\plauja}[1]{{{}_\Acal^{\Ucal_{#1}}\langle}}
\newcommand{\prauja}{\rangle}
\newcommand{\praujc}[1]{{\rangle_\Ccal^{\Ucal_{#1}}}}
\newcommand{\plaujc}{\langle}
\newcommand{\prim}{\operatorname{Prim}}
\newcommand{\pmu}{{p^{-1}}}
\newcommand{\qmu}{{q^{-1}}}
\newcommand{\ra}{\rangle}
\newcommand{\raa}{{\rangle}}
\newcommand{\rab}{{\rangle_\Bcal}}
\newcommand{\rabe}{\rangle}
\newcommand{\rac}{{\rangle_\Ccal}}
\newcommand{\rak}{\rangle}
\newcommand{\rmu}{{r^{-1}}}
\newcommand{\smu}{{s^{-1}}}
\newcommand{\supp}{\operatorname{supp}}
\newcommand{\tmu}{{t^{-1}}}
\newcommand{\vep}{\varepsilon}

\newcommand{\spec}[1]{\hat{#1}}

\newcommand{\spn}{\operatorname{span}}
\newcommand{\spncl}{\overline{\spn}\, }

\newcommand{\Acal}{\mathcal{A}}
\newcommand{\A}{\mathcal{A}} 
\newcommand{\Bcal}{\mathcal{B}}
\newcommand{\B}{\mathcal{B}} 
\newcommand{\Ccal}{\mathcal{C}}
\newcommand{\Dcal}{\mathcal{D}}

\newcommand{\Hcal}{\mathcal{H}}

\newcommand{\Ncal}{\mathcal{N}}

\newcommand{\Ucal}{\mathcal{U}}
\newcommand{\Vcal}{\mathcal{V}}

\newcommand{\Xcal}{\mathcal{X}}
\newcommand{\X}{\mathcal{X}} 
\newcommand{\Ycal}{\mathcal{Y}}
\newcommand{\Zcal}{\mathcal{Z}}

\newcommand{\Bb}{\mathbb{B}}

\newcommand{\kb}{\mathbbm{k}}
\newcommand{\Kb}{\mathbb{K}}
\newcommand{\Lb}{\mathbb{L}}

\newcommand{\Rb}{\mathbb{R}}

\newcommand{\LCH}{locally compact Hausdorff }

\newcommand*{\cstar}{\texorpdfstring{\(\mathrm{C}^*\)\nobreakdash-\hspace{0pt}}{*-}}

\newcommand*{\dd}{\,\mathrm d} 
\newcommand*{\red}{\mathrm r} 
\newcommand*{\contc}{\mathrm{C_c}} 
\newcommand*{\contz}{\mathrm{C_0}} 

\newcommand*{\nb}{\nobreakdash}
\newcommand*{\Star}{\(^*\)\nobreakdash-}
\newcommand*{\sbe}{\subseteq} 
\newcommand*{\braket}[2]{\langle#1{\mid}#2\rangle}
\newcommand*{\defeq}{\mathrel{\vcentcolon=}}

\theoremstyle{plain}
\newtheorem{theorem}{Theorem}[section]
\newtheorem{lemma}[theorem]{Lemma}
\newtheorem{corollary}[theorem]{Corollary}
\newtheorem{proposition}[theorem]{Proposition}
\theoremstyle{definition}
\newtheorem{definition}[theorem]{Definition}
\newtheorem{notation}[theorem]{Notation}
\theoremstyle{remark}
\newtheorem{remark}[theorem]{Remark}
\newtheorem{example}[theorem]{Example}

\numberwithin{equation}{section}

\title{Morita enveloping Fell bundles}
\author{Fernando Abadie}
\email{fabadie@cmat.edu.uy}
\address{Centro de Matemática\\
Facultad de Ciencias\\
Universidad de la República\\
Iguá 4225, 11400, Montevideo\\ 
Uruguay.}

\author{Alcides Buss}
\email{alcides@mtm.ufsc.br}
\address{Departamento de Matem\'atica\\
 Universidade Federal de Santa Catarina\\
 88.040-900 Florian\'opolis-SC\\
 Brazil}

\author{Dami\'{a}n Ferraro}
\email{dferraro@unorte.edu.uy}
\address{Departamento de Matemática \\
y Estadística del Litoral\\
Universidad de la República \\
Rivera 1350, 50000, Salto\\
Uruguay}

\thanks{The second named author is supported by CNPq.}

\makeindex
\begin{document}
\begin{abstract}
We introduce notions of weak and strong equivalence for non-saturated Fell bundles over locally compact groups and show that every Fell bundle is strongly (resp. weakly) equivalent to a semidirect product Fell bundle for a partial (resp. global) action.
Equivalences preserve cross-sectional \cstar{}algebras and amenability.
We use this to show that previous results on crossed products and amenability of group actions carry over to Fell bundles.
\end{abstract}

\maketitle

\tableofcontents

\section{Introduction}

A Fell bundle over a locally compact group $G$ is a continuous bundle
$\A\to G$ of Banach spaces $(A_t)_{t\in G}$ together with continuous
multiplications $A_s\times A_t\to A_{st}$, $(a,b)\mapsto a\cdot b$,
and involutions $A_s\to A_{s^{-1}}$, $a\mapsto a^*$, satisfying
properties similar to those valid for a \cstar{}algebra, like the
positivity $a^*\cdot a\geq 0$ and the \cstar{}axiom $\|a^*\cdot
a\|=\|a\|^2$. Fell bundles generalise partial actions of
groups. Indeed, in \cite{Exel:TwistedPartialActions} Exel defines
\emph{twisted partial actions} of groups on \cstar{}algebras and to
each such action, a Fell bundle is constructed, the so-called
semidirect product of the twisted partial action. This is, of course,
a generalisation of the semidirect product construction for ordinary
(global) actions already introduced by Fell in his first papers on the
subject, see \cite{Fell:extension-Mackey1,Fell:extension-Mackey,Fell:induced-reps,Doran-Fell:Representations,Doran-Fell:Representations_2}. In
\cite{Exel:TwistedPartialActions} only \emph{continuous} twisted partial
actions are considered, but even the measurable twisted actions of
Busby-Smith \cite{Busby-Smith:Representations_twisted_group} can be
turned into (continuous) Fell bundles, see
\cite{Exel-Laca:Continuous_Fell}.

The main result of \cite{Exel:TwistedPartialActions} already indicates
that Fell bundles are very close to partial actions: it asserts that
every \emph{regular} Fell bundle is isomorphic to a semidirect product
by a twisted partial action. It is moreover shown that every separable
Fell bundle can be ``regularised'' via stabilisation. The
stabilisation procedure should be viewed as a form of producing an
``equivalent'' Fell bundle in the spirit of Morita equivalence of
\cstar{}algebras. Hence we may say that the results in
\cite{Exel:TwistedPartialActions} show that every separable Fell
bundle is equivalent to one associated to a twisted partial action. A
very basic question appears: can the twist be ``removed'', that is, is
every Fell bundle equivalent to a partial action semidirect product
Fell bundle? For saturated Fell bundles this, indeed, follows from the
famous Packer-Raeburn Stabilisation Trick which asserts that every
twisted (global) action is stably isomorphic to an untwisted
action. As a result, every saturated separable Fell bundle is
equivalent to one coming from an ordinary action. This version is also
known for (separable) saturated Fell bundles over groupoids as proved
in
\cites{Buss-Meyer-Zhu:Higher_twisted,Ionescu-Kumjian-Sims-Williams:Stabilization},
where precise notions of equivalence of saturated Fell bundles are
introduced. A version of the stabilisation trick for non-saturated
Fell bundles is only known for discrete groups: it is proved in the
master thesis of Sehnem \cite{Sehnem:Master} and reproduced in Exel's
book \cite{Exel:Partial_dynamical} that, after stabilisation, every
separable Fell bundle becomes isomorphic to one coming from an
(untwisted) partial action.

In \cite{Buss-Meyer-Zhu:Higher_twisted} a new point of view is introduced from which saturated Fell bundles are interpreted as actions of the underlying group(oid) in the bicategory of \cstar{}correspondences. Indeed, the algebraic structure of the Fell bundle can be used to turn each fibre $A_t$ into a Hilbert bimodule over the unit fiber \cstar{}algebra $A\defeq A_e$, and these bimodules are imprimitivity (or equivalence) bimodules if (and only if) the Fell bundle is saturated. Hence we may view a saturated Fell bundle as an action of $G$ on $A$ by equivalences. A non-saturated Fell bundle should be viewed as a partial action of $G$ on $A$ by (partial) equivalences.

Although the notion of equivalence between saturated Fell bundles over
groups is already well established nowadays, little is known for
non-saturated Fell bundles. Only recently a more general notion of
equivalence has been introduced in \cite{Abadie-Ferraro:equivFB}, which in the present
work we call \emph{weak equivalence}. This notion originates in
\cite{Abadie:Enveloping}: the relationship between the Fell bundles of
a partial action and its enveloping action is precisely that of weak
equivalence. We introduce yet another notion
of equivalence, the \emph{strong equivalence}. As the name suggests,
strong equivalence is stronger than weak equivalence. Strong
equivalence is the natural extension
of the notion of (Morita) equivalence for partial actions as
introduced by the first named author in
\cite{Abadie:Enveloping}. Indeed, we are going to extend one of the
main results in \cite{Abadie:Enveloping} and prove that every (not
necessarily saturated or separable) Fell bundle over $G$ is strongly
equivalent to a semidirect product Fell bundle by a partial action of
$G$ (Theorem~\ref{theorem every Fell bundle is strongly equivalent to a semidirect product bundle}).
On the other hand we will prove that, as long as saturated Fell
bundles are concerned, there is no difference between weak equivalence
and strong equivalence of Fell bundles (Corollary~\ref{cor:sats}); they extend the usual notion of equivalence for global actions.

In the recent paper \cite{kwasniewski-Meyer:Fell} by Kwa\'sniewski and Meyer the notion of
  \emph{Morita globalization} of a Fell bundle over a discrete group is introduced, and it is
shown that every Fell bundle (over a discrete group) has a Morita globalization. It can be
shown that the Fell bundle associated to the action involved in the
definition of a Morita globalization of a Fell bundle is
weakly equivalent, in our sense, to the original Fell bundle. As
a result a Morita globalization of a Fell bundle is an instance of what we call here a
\emph{Morita enveloping Fell bundle} (see Definition~\ref{def:Morita-enveloping}).

The notion of weak equivalence will allow us to show that every
partial action of $G$, once viewed as a Fell bundle, is weakly
equivalent to a global action. As a conclusion, every Fell bundle is
weakly equivalent to one associated to a global action. We shall prove
this in one step, showing directly that the Fell bundle is weakly
equivalent to the semidirect product Fell bundle of a global
action. This global action is directly constructed from the Fell
bundle. Indeed, it is the same action $\alpha$ appearing in
\cite{Abadie:Enveloping} which takes place on the \cstar{}algebra of
kernels $\kA$ of the Fell bundle $\A$. As already shown in
\cite{Abadie:Enveloping}, $\kA$ can be canonically identified with the
crossed product $C^*(\A)\rtimes_{\delta_\A}G$ by the dual coaction
$\delta_\A$ of $G$ on the full cross-sectional \cstar{}algebra
$C^*(\A)$ of $\A$; one can also use the reduced \cstar{}algebra
$C^*_\red(\A)$ together with its dual coaction $\delta_\A^\red$ of
$G$, which is a normalisation of $\delta_\A$. Since the dual coaction
$\delta_\A$ is maximal, $\kA\rtimes_\alpha G\cong
C^*(\A)\otimes\Kb(L^2(G))$ and similarly
$\kA\rtimes_{\alpha,\red}G\cong C^*_\red(\A)\otimes\Kb(L^2(G))$. As a
consequence of this, the notion of weak (hence strong) equivalence
preserves full and reduced cross-sectional \cstar{}algebras, that is,
weakly equivalent Fell bundles have (strongly Morita) equivalent full
and reduced cross-sectional \cstar{}algebras. Using the same idea, we
also derive a version of this result for certain exotic completions
$C^*_\mu(\A)$ introduced in \cite{Buss-Echterhoff:Maximality} (some of
the latter results were also obtained in
\cite{Abadie-Ferraro:equivFB}, though with different methods).
Moreover, we show that two Fell bundles $\A$ and $\B$ are weakly equivalent if
and only if the corresponding actions on their \cstar{}algebras of kernels $\kA$ and $\kB$ are equivariantly Morita equivalent,
if and only if their dual coactions on $C^*(\A)$ and $C^*(\B)$ are equivariantly Morita equivalent.
Strong equivalence of Fell bundles can also be characterised in a similar fashion by the restriction of the global actions on $\kA$ and $\kB$ to
the partial actions on the \cstar{}algebras of compact operators $\Kb(L^2(\A))$ and $\Kb(L^2(\B))$.
\par Section~\ref{sec:partial actions} of the paper can be viewed as a sample of
the potential applications of our main results from the previous sections.
We study the partial action associated to Fell bundles on spectra level: a Fell bundle
$\Acal$ over $G$ induces a partial action of $G$ on the spectrum (both primitive and irreducible representations) of the unit fibre $A_e$.
This partial actions have already been introduced in \cite{Abadie-Abadie:Ideals} for discrete groups. We extend the construction to all locally compact groups, proving that the partial action is always continuous. We then show that the enveloping action of the spectral partial action associated to a Fell bundle $\A$ is precisely the global action on the spectrum of $\kA$ induced by its canonical action $\alpha$.
In particular, many results of \cite{Abadie-Abadie:Ideals} can be obtained from the already existing results
for crossed products by ordinary actions. In the same spirit we extend some results
about amenability and nuclearity of crossed products to the
realm of Fell bundles.

We also add an appendix where we use the tensor product construction of equivalence bundles from \cite{Abadie-Ferraro:equivFB} to prove that strong equivalence of Fell bundles is an equivalence relation.

\section{Fell bundles and globalization of weak group partial
  actions}\label{sec:fb and glob}

Let $\Bcal$ be a Fell bundle over a locally compact group $G$, both fixed for the rest of this article.
We denote by $\dd t$ the integration with respect to a fixed left invariant Haar measure on $G$ and write $\Delta$ for the modular function of $G$.

\begin{notation}\label{notation of product of sets}
  Given two sets $X$ and $Y$ for which a product $xy$ between elements $x\in X$ and $y\in Y$ is defined and is contained in a normed vector space, we write $XY$ to mean the closed linear space of all such products, that is,
  $$XY\defeq \spncl \{xy\colon x\in X,\ y\in Y\}.$$
  This applies, for instance, if $X$ and $Y$ are subsets of two fibers $B_r$ and $B_s$ of a Fell bundle $\B$, in which case $XY$ is a closed linear subspace of $B_{rs}$.
\end{notation}

Following \cite{Buss-Meyer-Zhu:Higher_twisted}, we view a saturated Fell bundle $\Bcal$ as an \emph{action by equivalences}
of $G$ on the \cstar{}algebra $B_e$ (the unit fiber). A non-saturated
Fell bundle is viewed as a \emph{partial action by equivalences} of
$G$ on $B_e$.

To explain this idea explicitly take $t\in G$ and define $D^\Bcal_t:=B_tB_t^*=B_t B_\tmu$ and note that $B_t$ is a $D^\Bcal_t-D^\Bcal_\tmu$\nb-equivalence bimodule with the natural structure inherited from $\Bcal$.
The key is to think of $B_t$ as an arrow from $D^\Bcal_\tmu$ to $D^\Bcal_t:$
$$\xymatrix{ D^\Bcal_t   & D^\Bcal_\tmu \ar[l]_{B_t} },$$
where the arrows go from right to left to be consistent with the example below.

Given $r,s\in G$, the $B_e$\nb-tensor product $B_r\otimes_{B_e}B_s$ is the composition of arrows and implements a Morita equivalence between $J_{r,s}:=B_rB_s(B_rB_s)^*$ and $J_{\smu,\rmu}$.
Moreover, $B_r\otimes_{B_e}B_s$ is isomorphic to $J_{r,s}B_{rs}=B_{rs}J_{\smu,\rmu}$ through the unique unitary $U$ such that
$$U\colon B_r\otimes_{B_e}B_s\to J_{r,s}B_{rs},\ a\otimes b\mapsto ab. $$
Then the composition of $B_r$ with $B_s$, namely $ B_r\otimes_{B_e}B_s$, is (isomorphic to) a restriction of $B_{rs}$ to an ideal.

In this way every Fell bundle becomes a \cstar{}partial action by equivalence bimodules and the bundle is saturated if and only if the action is global (meaning that $D_t^\Bcal=D_e^\Bcal=\Bcal_e$ for all $t\in G$).

\begin{example}
 Let $\al=\left(\{A_t\}_{t\in G},\{\al_t\}_{t\in G}\right)$ be a partial action of $G$ on the \cstar{}algebra $A$ and let $\Bcal_\al$ be its semidirect product bundle \cite{Exel:TwistedPartialActions}.
 The fiber $B_t$ is $A_t\times \{t\}=A_t\delta_t$ and $D^{\Bcal_\al}_t = A_t\times \{e\}\cong A_t$.
 The operations on the equivalence $A_t-A_{t^{-1}}$-bimodule $B_t$ are given by
 \begin{align*}
   {}_{A_t}\la x\delta_t ,y\delta_t\ra  &= xy^* &  a\cdot x\delta_t &= ax\delta_t\\
   \la x\delta_t,y\delta_t\ra_{A_\tmu} & = \al_\tmu(x^*y) &   x\delta_t\cdot b &= \al_t(\al_\tmu(x)b)\delta_t.
 \end{align*}
\end{example}

Next we extend the notion of equivalence between partial actions to the context of Fell bundles.

\begin{definition}
Let $\Bcal$ be a Fell bundle over $G$. A \emph{right Hilbert $\Bcal$-bundle} is a Banach bundle $\Xcal$ over $G$ with continuous functions
 \begin{equation}\label{R bundle operations}
  \lab \ ,\ \rab\colon  \Xcal\times \Xcal\to \Bcal,\ (x,y)\mapsto \lab x,y\rab,\qquad \Xcal\times \Bcal\to \Xcal,\ (x,b)\mapsto xb,
 \end{equation}
 such that:
 \begin{enumerate}[(1R)]
  \item For all $r,s\in G$, $X_rB_s\sbe X_{rs}$ and $\lab X_r,X_s\rab\sbe B_{\rmu s}$.
  \item For all $r,s\in G$ and $x\in X_r$ the function $X_r\times B_s\to X_{rs}$, $(x,b)\mapsto xb$, is bilinear and $X_s\to B_{\rmu s}$, $y\mapsto \lab x,y\rab$, is linear.
  \item For all $x,y\in X$ and $b\in B$, $\lab x,y\rab^* =\lab y,x\rab$, $\lab x,yb\rab  =\lab x,y\rab b$, $\lab x,x\rab\geq 0$ (in $B_e$) and $\|x\|^2 = \|\lab x,x\rab\|$.
 \end{enumerate}
We say that $\Xcal$ is \emph{full} if
\begin{equation}\label{eq:full-condition}
\spncl\{\lab X_r,X_r\rab\colon r\in G\}=B_e.
\end{equation}
We say that $\Xcal$ is \emph{strongly full} if
\begin{equation}\label{eq:strong-full-condition}
\spncl \lab X_r,X_r\rab=B_r^*B_r\quad\mbox{ for all }r\in G.
\end{equation}
\end{definition}

\begin{remark}\label{remark density of inner products}
(1) By a Banach bundle we mean a \emph{continuous} Banach bundle in the sense of Doran-Fell, see \cite{Doran-Fell:Representations}.
In particular, a Fell bundle is a continuous Banach bundle, by definition. However, the main axiom concerning the continuity of the bundle, namely,
the continuity of the norm function $\Bcal\to [0,\infty)$, $b\mapsto \|b\|$, is somehow automatic, see \cite{Buss-Meyer-Zhu:Higher_twisted}*{Lemma 3.16}. A similar observation holds for every Hilbert $\Bcal$-bundle: the continuity of the norm function $x\mapsto \|x\|$ on $\Xcal$ follows from the continuity of the norm function on $\B$ because $\|x\|=\|\lab x,x\rab\|^{1/2}$.

(2) The fullness condition~\eqref{eq:full-condition} is equivalent to
  the condition that $$B_r=\spncl\{\la X_s,X_{sr}\rab\colon s\in
  G\}=\spncl\{\la X_s,X_t\rab: s^{-1}t=r\}$$ for all $r\in G$ because
  if~\eqref{eq:full-condition} holds, then
  $$B_r = B_eB_r=\spncl \{\la X_s,X_s\rab B_r\colon s\in G\}\sbe
  \spncl\{\la X_s,X_{sr}\rab\colon s\in G\}\sbe B_r.$$
In general, $\lab X_r,X_r\rab$ is only contained in $B_e$, not
necessarily in the ideal $B_r^*B_r\sbe B_e$.
Hence the strong fullness
condition~\eqref{eq:strong-full-condition} requires that
$\lab X_r,X_r\rab$ is contained and is linearly dense in the ideal
$B_r^*B_r$. Moreover, if $\Xcal$ is strongly full, then $\spncl \lab X_r,X_s\rab = B_r^*B_s$ for all $r,s\in G$ because:
\begin{align*}
\spncl  \lab X_r,X_s\rab
    & =\spncl \lab X_r\lab X_r,X_r\rab,X_s\lab X_{s},X_s\rab\rab \\
    & =  \subseteq \spncl \lab X_r,X_r\rab\lab X_r,X_s\rab\lab X_s,X_s\rab\\
    &  \subseteq \spncl B_r^*(B_r\lab X_r,X_s\rab)B_s^*B_s
      \subseteq \spncl B_r^*B_sB_s^*B_s  = B_r^*B_s
\end{align*}
and
\begin{align*}
B_r^*B_s
  & =  B_r^* (B_\rmu)^* B_\rmu B_sB_\smu B_s
    = \spncl B_r^* \lab X_\rmu,X_\rmu\rab B_s\lab X_s,X_s\rab\\
  & = \spncl \lab X_\rmu B_r,X_\rmu\rab\lab X_s B_s^*,X_s\rab
    \subseteq \spncl \lab X_e,X_\rmu\rab\lab X_e,X_s\rab\\
  & \subseteq  \spncl  \lab X_e\lab X_\rmu,X_e\rab,X_s\rab
    \subseteq  \spncl  \lab X_r,X_s\rab.
\end{align*}
\end{remark}

Left Hilbert bundles are similarly defined. We spell out the complete definition for convenience.

\begin{definition}
Let $\Acal$ be a Fell bundle over $G$. A left Hilbert $\Acal$-bundle is a Banach bundle $\Xcal$ over $G$ with continuous functions
  \begin{equation}\label{L bundle operations}
  \laa \ ,\ \raa \colon  \Xcal\times \Xcal\to \Acal,\ (x,y)\mapsto \laa x,y\raa,\qquad \Acal\times \Xcal\to \Xcal,\ (a,x)\mapsto ax,
 \end{equation}
 such that:
 \begin{enumerate}[(1L)]
  \item For all $r,s\in G$, $A_rX_s\sbe X_{rs}$ and $\laa X_r,X_s\raa\sbe A_{r\smu}$.
  \item For all $r,s\in G$ and $x\in X_r$ the function $A_r\times X_s\to X_{rs}$, $(a,x)\mapsto ax$, is bilinear and $X_s\to A_{s\rmu}$, $y\mapsto \laa y,x\raa$, is linear.
  \item For all $x,y\in X$ and $a\in A$, $\laa x,y\raa^* =\laa y,x\raa$, $\laa ax,y\raa  =a\laa x,y\raa $, $\laa x,x\raa\geq 0$ (in $A_e$) and $\|x\|^2 = \|\laa x,x\raa\|$.
 \end{enumerate}
If
\begin{equation}
A_e=\spncl\{\laa X_r,X_r\raa\colon r\in G\},
\end{equation}
$\Xcal$ is called \emph{full}, and if
 \begin{equation}
A_rA_r^* =\spncl \laa X_r,X_r\raa\quad\mbox{ for all }r\in G,
 \end{equation}
$\Xcal$ is called \emph{strongly full}.
\end{definition}

\begin{definition}
Let $\Acal$ and $\Bcal$ be Fell bundles over $G$.
A \emph{weak $\Acal-\Bcal$-equivalence bundle} is a Banach bundle $\Xcal$ which is a full left Hilbert $\Acal$\nb-bundle, a full right Hilbert $\Bcal$\nb-bundle and $\laa x,y\raa z=x\lab y,z\rab$ for all $x,y,z\in \Xcal$. In this case we say that $\Acal$ and $\Bcal$ are \emph{weakly equivalent}. If, in addition, $\Xcal$ is strongly full, both as a left and right bundle, we say that $\Xcal$ is a \emph{strong $\Acal-\Bcal$-equivalence} and that $\Acal$ and $\Bcal$ are \emph{strongly equivalent}.
\end{definition}

We have included an appendix where we show several properties regarding tensor products of equivalence bundles.
For example we show, in Theorem \ref{theorem:strong equivalence is an equivalence relation}, that strong equivalence is an equivalence relation.
Weak equivalence was shown to be an equivalence relation in \cite{Abadie-Ferraro:equivFB}.

\begin{example}\label{ex:moritaeqpa}
 Every equivalence of partial actions (see \cite{Abadie:Enveloping}*{Section~4.2} and \cite{Exel:Partial_dynamical}*{Definition~15.7}) can be turned into a strong equivalence between the associated Fell bundles. Suppose
 $\alpha=\{I_{t^{-1}}\stackrel{\alpha_t}{\to}I_t\}_{t\in G}$ and
 $\beta=\{J_{t^{-1}}\stackrel{\beta_t}{\to}J_t\}_{t\in G}$ are partial
 actions on the \cstar{}algebras $A$ and $B$ respectively, and suppose $X$ is an
 $A-B$\nb-equivalence bimodule such that $I_tX=XJ_t$ for all $t\in
 G$. For $t\in G$, define $X_t:=I_tX=XJ_t$ and suppose
 $\gamma=\{X_{t^{-1}}\stackrel{\gamma_t}{\to}X_t\}_{t\in G}$ is a
 partial action of $G$ on $X$ such that
\[\alpha_t(\langle x,y\rangle_A)\gamma_t(z)
=\gamma_t(\langle x,y\rangle_Az)
=\gamma_t(x\langle y,z\rangle_B)
=\gamma_t(x)\beta_t(\langle y,z\rangle_B)\]
for all $t\in G$ and $x,y,z\in X_t$.
Then $\alpha$ and $\beta$ are said to be Morita equivalent, and the
following notations are used: $X^l:=A$, $X^r:=B$, $\gamma^l=\alpha$,
and $\gamma^r=\beta$ (in fact $X$ determines $A$ and $B$ up to
isomorphism, and then $\gamma^l$ and $\gamma^r$ are determined by
$\gamma$; see \cite{Abadie:Enveloping} for details).
 Let $\Lb(\gamma)$ be the linking partial action of $\gamma$ (see the proof of  Proposition 4.5 of \cite{Abadie:Takai_crossed}) and let $\Bcal_{\Lb(\gamma)}$ be the Fell bundle associated with $\Lb(\gamma)$. Define $\Xcal_\gamma$ as the Banach subbundle of $\Bcal_{\Lb(\gamma)}$
 $$X_\gamma:=\left\{ \left(\begin{array}{cc}
                        0 & x\\ 0 & 0
                      \end{array}
 \right)\delta_t\colon x\in X_t,\ t\in G \right\}.$$

 With the structure inherited from the identity $\Bcal_{\Lb(\gamma)}-\Bcal_{\Lb(\gamma)}$-bundle structure of $\Bcal_{\Lb(\gamma)}$, $\Xcal_\gamma$ is a strong $\Bcal_\al-\Bcal_\be$-equivalence bundle, where $\Bcal_\al$ and $\Bcal_\be$ denote the Fell bundles associated with $\alpha$ and $\beta$, respectively.
\end{example}

The notion of weak equivalence allows us to ``identify'' partial
actions with the corresponding enveloping actions, in case these
exist.
This is explained in the following example. In particular this shows
that a non-saturated Fell bundle may be weakly equivalent to a
saturated one.

\begin{example}\label{example globalization}
  Let $\be$ be a global action of $G$ on the \cstar{}algebra $B$ and assume that $A$ is a \cstar{}ideal of $B$
  such that $B=\spncl\{\be_t(A)\colon t\in G\}$. This means that $\beta$ is the enveloping (global) action of the partial action
  $\al$ given as the restriction of $\be$ to $A$ (see \cite{Abadie:Enveloping}). In this situation, $\Bcal_\al$ is weakly equivalent to $\Bcal_\be$.
  The equivalence is implemented by the bundle $\Xcal=A\times G$, considered as a Banach subbundle of $\Bcal_\be$ and viewing $\Bcal_\al$ as a Fell subbundle of $\Bcal_\be$. The operations are the ones inherited from the identity $\Bcal_\be-\Bcal_\be$-bundle. Notice that $\Bcal_\al$ is, in general, not strongly equivalent to $\Bcal_\be$ because a strong equivalence between Fell bundles implies in a (strong) Morita equivalence between their unit fibers $A$ and $B$. And it is easy to produce examples where this is not the case. For instance, one may take a commutative \cstar{}algebra $B=\contz(X)$ and an ideal $A\sbe B$ which is not isomorphic to $B$, like $B=\contz(\Rb)$ and $A=\contz((0,1)\cup (1,2))$ with $G=\Rb$ acting by translation.
\end{example}

With notation as in Example~\ref{example globalization}, we have $\Bcal_\al\Xcal\sbe \Xcal$, $\Xcal\Bcal_\be=\Xcal$, $\Bcal_\al=\Xcal\Xcal^*$ and $\Xcal^*\Xcal=\Bcal_\be$ (where, for example, the equality $\Bcal_\al=\Xcal\Xcal^*$ means that the $t$-fiber of $\Bcal_\al$ is the closed linear span of all $\Xcal_s\Xcal_r^*$ with $sr^{-1}=t$). This motivates the following.

\begin{definition}
  An enveloping bundle of a Fell bundle $\Acal$ is a saturated Fell bundle $\Bcal$ for which there exists a Fell subbundle $\Ccal\sbe \Bcal$ and an isomorphism of Fell bundles $\pi\colon \Acal\to \Ccal$ such that for $\Xcal:=\Ccal\Bcal$, we have $\Xcal\Xcal^*=\Ccal$ and $\Xcal^*\Xcal=\Bcal$.
\end{definition}

\begin{remark}
  With notation as above, the bundle $\Xcal$ above is a weak equivalence $\Acal-\Bcal$\nb-bundle with the operations $\laa x ,y\raa = \pi^{-1}(xy^*)$, $(a,x)\mapsto \pi(a)x$, $\lab x,y\rab = x^*y$ and $ (x,b)\mapsto xb$. Hence every Fell bundle is weakly equivalent to its enveloping bundle (if it admits one).
  The equivalence is, however, not strong in general (see Example \ref{example globalization}).
\end{remark}

Imitating the notion of Morita enveloping action from \cite{Abadie:Enveloping} we state the following.

\begin{definition}\label{def:Morita-enveloping}
  A Morita enveloping bundle of a Fell bundle $\Acal$ is a saturated Fell bundle $\Bcal$ which is the enveloping bundle of a Fell bundle strongly equivalent to $\Acal$.
\end{definition}

It is shown in \cite{Abadie:Enveloping} that every partial action on a \cstar{}algebra has a Morita enveloping action.
In the next section we show that every Fell bundle admits a Morita enveloping Fell bundle. Moreover, we show that this Morita enveloping Fell bundle can be realised as a semidirect product bundle of a global action. This global action is unique up to Morita equivalence of actions on \cstar{}algebras.

\begin{remark}\label{rem:weak uniqueness}
Since weak equivalence of Fell bundles is an equivalence
relation, the Morita enveloping bundle of a Fell bundle is unique up
to weak equivalence. In fact, we will show in
Corollary~\ref{cor:strong uniqueness} that it is unique up to strong
equivalence.
\end{remark}

\subsubsection*{{\bf The bundle of generalized compact operators.}} Given a full
right Hilbert $\Bcal$-bundle $\Xcal$ there exists, up to isomorphism,
a unique Fell bundle $\Kb(\Xcal)$ such that $\Xcal$ is a weak
$\Kb(\Xcal)-\Bcal$ equivalence bundle. We recall next the
main lines of the construction of $\Kb(\Xcal)$, and we refer
to \cite{Abadie-Ferraro:equivFB} for complete details.

To describe the fiber
over $t\in G$ of the bundle $\Kb(\Xcal)$, note first that,
given $x,y\in \Xcal$, say $x\in X_{ts}$ and $y\in X_s$ for some $s,t\in
G$,  we have a map $[x,y]:\Xcal\to\Xcal$ such that
$[x,y]z:=x\langle y,z\rangle$ for all $z\in \Xcal$. The
map $[x,y]$ has the following properties:
\begin{enumerate}
  \item $[x,y]X_r\subseteq X_{tr}$ for all $r\in G$
  \item $[x,y]$ is linear when restricted to each fiber $X_r$ of~$\Xcal$.
  \item $[x,y]$ is continuous.
  \item $[x,y]$ is bounded: its norm $\|[x,y]\|:=\sup_{\{z\in\Xcal:\|z\|\leq
      1\}}\|[x,y]z\|$ is finite with
    $\|[x,y]\|\leq\|x\|\,\|y\|$.
  \item $[x,y]$ is adjointable: there exists a (necessarily unique) \emph{adjoint} operator $[x,y]^*:\Xcal\to \Xcal$ such
    that $\langle[x,y]z,z'\rangle=\langle z,[x,y]^*z'\rangle$ for all $z,z'\in\Xcal$. Moreover, we have $[x,y]^*=[y,x]$.
\end{enumerate}
It is not hard to check that the vector space $\Bb_t(\Xcal)$
of maps $S:\Xcal\to\Xcal$ that satisfy properties (1)--(5) above is a
Banach space, in fact a \cstar{}ternary ring with the operation
$(S_1,S_2,S_3):=S_1S_2^*S_3$. The elements of $\Bb_t(\Xcal)$
are called adjointable operators of order $t$. If $G_d$ is the group
$G$ with the discrete topology, it follows that the family
$(\Bb_t(\Xcal))_{t\in G}$ is a Fell bundle over $G_d$, where
the product is given by composition. Now define
$\Kb_t(\Xcal):=\overline{\textrm{span}}\{[x,y]:\, x\in X_{ts},
y\in X_s, s\in G\}$. It is easy to check that
$\Kb(\Xcal):=(\Kb_t(\Xcal))_{t\in G}$ is a Fell
subbundle of $(\Bb_t(\Xcal))_{t\in G}$. Finally, there is a
suitable topology on $\Kb(\Xcal)$ making it a Fell bundle over
$G$, and $\Xcal$ is a weak $\Kb(\Xcal)-\Bcal$ equivalence
bundle with the obvious operations and inner products, see \cite{Abadie-Ferraro:equivFB} for details.

The Fell bundle $\Kb(\Xcal)$ is unique in the
following sense: if $\Xcal$ is a weak $\Acal-\Bcal$ equivalence, then
there exists an isomorphism $\pi:\Acal\to\Kb(\Xcal)$ such that
$\pi({}_{\Acal}\langle x,y\rangle)=[x,y]$ for all $x,y\in \Xcal$ (see
\cite[Corollary~3.10]{Abadie-Ferraro:equivFB}).
\subsubsection*{{\bf The linking Fell bundle of an equivalence
    bundle.}}
Given a weak $\Acal-\Bcal$ equivalence $\Xcal$, it is possible to
define a Fell bundle $\Lb(\Xcal)=(L_t)_{t\in G}$ which plays a
role similar to that of the linking algebra of an
imprimitivity bimodules. The fiber $L_t$ over $t\in G$ is defined to
be $L_t:=\begin{pmatrix}
  A_t&X_t\\\tilde{X}_{t^{-1}}&B_t\end{pmatrix}$ with entrywise vector
space operations (here, given an $A-B$
Hilbert bimodule $X$, $\tilde{X}$ denotes its dual $B-A$ Hilbert
bimodule). The operations and topology on $\Lb(\Xcal)$ are
defined as follows:
  \begin{enumerate}
  \item Product and involution on $\Lb(\Xcal)$ are given by
  $$ \begin{pmatrix}
    a & x\\
    \widetilde{y} & b
   \end{pmatrix}
    \begin{pmatrix}
    c & u\\
    \widetilde{v} & d
      \end{pmatrix}
    =\begin{pmatrix}
    ac + \laa x,v\raa & au+xd\\
    \widetilde{c^*y} + \widetilde{vb^*} & \lab y,u\rab + bd
   \end{pmatrix} \quad \mbox{and}
   $$
   $$
   \begin{pmatrix}
    a & x\\
    \widetilde{y} & b
   \end{pmatrix}^*
    =\begin{pmatrix}
    a^* & y\\
    \widetilde{x} & b^*
   \end{pmatrix}.$$
   \item Given $\xi\in C_c(\Acal)$, $\eta\in C_c(\Bcal)$ and $f,g\in C_c(\Xcal)$ the function
   $$\begin{pmatrix}
    \xi & f\\
    g & \eta
   \end{pmatrix}\colon G\to \Lb(\Xcal), \ t\mapsto
   \begin{pmatrix}
    \xi(t) & f(t)\\
    \widetilde{g(\tmu)} & \eta(t)
   \end{pmatrix} $$
   is a continuous section (see \cite[13.18]{Doran-Fell:Representations}).
  \end{enumerate}

The subbundle $\Acal\oplus\Xcal$ of
$\mathbb{L}(\Xcal)$ is then a weak $\Acal-\mathbb{L}(\Xcal)$ equivalence
bundle, and the subbundle $\Xcal\oplus\Bcal$ is a weak
$\mathbb{L}(\Xcal)-\Bcal$ equivalence bundle. We refer the reader to
the third section of \cite{Abadie-Ferraro:equivFB} for details.
\section{Canonical action on the kernels and Morita equivalence}\label{sec:canonical-equivalence}

Recall that $L^2(\Bcal)$ is the (full) right Hilbert $B_e$\nb-module obtained as the completion of $\contc(\Bcal)$ with respect to the pre-Hilbert $B_e$\nb-module structure given by the operations
$$\la f,g\ra_{L^2}:=\int_G f(t)^*g(t)\dd t,\qquad (f\cdot b)(t):=f(t)b,$$
for $f,g\in \contc(\Bcal)$ and $b\in B_e$.

The Banach bundle $\LtwoB$ is, as a topological bundle, the constant fiber bundle
$$L^2(\Bcal)\times G\to G,\ f\delta_t\mapsto t.$$
The norm is given by $\|f\delta_r\|:=\Delta(r)^{-1/2}\|f\|_{L^2}$.

\begin{proposition}
 Given $r,s,t,p\in G$, $f,g\in \contc(\Bcal)$ and $b\in B_t$, define
 \begin{align}
  \lab f\delta_r,g\delta_s\rab &\defeq\int_G f(pr)^*g(ps)\, \dd p\label{inner product}\\
  f\delta_r b &\defeq fb\delta_{rt},\ \mbox{ with }fb(p)\defeq f(p\tmu)b\label{action on the left}.
 \end{align}
With these operations, $\LtwoB$ becomes a full right Hilbert $\Bcal$-bundle.
\end{proposition}
\begin{proof}
  To simplify the notation we define $L_r:=\contc(\Bcal)\times \{r\}\sbe \LtwoB$.
  It is clear that the function $L_r\times B_t\to L_{rt}$, $(f\delta_r,b)\mapsto f\delta_rb$, is bilinear and that $L_r\times L_s\to B_{\rmu s},\ g\delta_s\mapsto \lab f\delta_r,g\delta_s\rab$, is linear.
  Straightforward computations show that $\lab f\delta_r,g\delta_s b\rab = \lab f\delta_r,g\delta_s\rab b$, $\lab f\delta_r,g\delta_s b\rab^*=\lab g\delta_s b,f\delta_r\rab$ and $\lab f\delta_r,f\delta_r\rab = \Delta(r)^{-1}\la f,f\ra_{L^2}$. In particular, $\lab f\delta_r,f\delta_r\rab\geq 0$.

  The canonical pre-Hilbert $B_e$\nb-module structure of $L_r$ induces the norm of $\LtwoB$, because
  $$\|\lab f\delta_r,f\delta_r\rab\|= \|\Delta(r)^{-1}\la f,f\ra_{L^2}\|=\Delta(r)^{-1}\|f\|_{L^2}^2.$$

  The action of $\Bcal$ on $\contc(\Bcal)\times G$ can be extended in a unique way to $\LtwoB$ because
  $$ \|f\delta_r b\|^2=\|b\lab f\delta_r,f\delta_r\rab b\|\leq \|b\|^2\| \lab f\delta_r,f\delta_r\rab \| =\|b\|^2\|f\delta_r\|^2.$$

  To see that the inner product defined on $\contc(\Bcal)\times G$ extends to $\LtwoB$ it suffices to prove that
  \begin{equation}\label{equation inner product bounded}
    \|\lab f\delta_r,g\delta_s\rab\|\leq \|f\delta_r\|\|g\delta_s\|.
  \end{equation}
  To do this take a representation $T\colon \Bcal\to \Bb(\Hcal)$ with $T|_{B_e}$ faithful.
  Then $\|T_b\|=\|T_{b^*b}\|^{1/2}=\|b^*b\|^{1/2}=\|b\|$ for all $b\in \Bcal$.
  Let $Tf\delta_r\in \contc(G,\Bb(\Hcal))$ be defined as $Tf(t)=T_{f(tr)}$ and consider $\contc(G,\Bb(\Hcal))$ as a subspace of $L^2(G,\Bb(\Hcal))$.
  The Cauchy-Schwarz inequality in $L^2(G,\Bb(\Hcal))$ implies
  \begin{align*}
   \| \lab f\delta_r,g\delta_s\rab \|
    & = \left\|\int_G [Tf\delta_r(t)]^* Tg\delta_s(t)\dd t\right\|
      = \|\la Tf\delta_r(t),Tg\delta_s\ra\|\\
    & \leq  \|\la Tf\delta_r(t),Tg\delta_s\ra\|^{1/2}\|\la Tf\delta_r(t),Tg\delta_s\ra\|^{1/2}\\
    & \leq \|T_{\lab f\delta_r,f\delta_r\rab}\|^{1/2}\|T_{\lab g\delta_s,g\delta_s\rab}\|^{1/2}
      = \| f\delta_r \|\|g\delta_s\|.
  \end{align*}
  This implies inequality (\ref{equation inner product bounded}).

  With respect to the density of inner products note that
  $$\spncl \lab L_r,L_r\rab = \spncl \Delta(r)^{-1}\la \contc(\Bcal),\contc(\Bcal)\rab =B_e,$$
  for all $r\in G$.

  The constant section associated to $f\in L^2(\Bcal)$ is $f\delta\colon G\to \LtwoB$, $t\mapsto f\delta_t$.
  Since for all $r\in G$, $\{f\delta_r\colon f\in \contc(\Bcal)\}=L_r$, to show that the inner product and action are continuous it suffices to prove that the functions
  $$ G\times G\to \Bcal,\ (r,s)\mapsto \int_G f(tr)^*g(ts)\dd t, \quad \mbox{ and }\quad  G\times G\to L^2(\Bcal), (r,s)\mapsto f[g(r)], $$
  are continuous for all $f,g\in \contc(\Bcal)$.
  The continuity of the first function follows adapting \cite[II 15.19]{Doran-Fell:Representations}.
  The other function has range in $\contc(\Bcal)$ and is continuous in the inductive limit topology, so it is continuous as a function with codomain $L^2(\Bcal)$.
\end{proof}

\begin{definition}
  The \emph{canonical $L^2$\nb-bundle} of the Fell bundle $\Bcal$ is
  the Hilbert $\Bcal$\nb-bundle $\LtwoB$ described in the last
  Proposition.
\end{definition}

We are interested in the identification of the Fell bundle of
generalized compact operators $\Kb(\LtwoB)$ of $\LtwoB$ (see end of
Section~\ref{sec:fb and glob}), up to
isomorphism of Fell bundles, because this Fell bundle is
weakly equivalent to $\Bcal$. We will show that $\Kb(\LtwoB)$ is a semidirect
product Fell bundle associated to an action of $G$ on a
\cstar{}algebra.

Following \cite{Abadie:Enveloping}, we write $\kcB$ for the space of
compactly supported continuous functions $k\colon G\times G\to \Bcal$
with $k(r,s)\in B_{r\smu}$ for all $r,s\in G$. In other words, $\kcB$
is the space of compactly supported continuous sections of the
pullback of $\Bcal$ along the map $G\times G\to G$, $(r,s)\mapsto
r\smu$. It is a normed \Star{}algebra with
$$ h*k(r,s)=\int_G h(r,t)k(t,s)\dd t \qquad k^*(r,s)=k(s,r)^*$$
$$\|k\|_2:=\left(\int_{G^2}\|k(r,s)\|^2\dd r\dd s\right)^{1/2}.$$
We may also endow $\kcB$ with the inductive limit topology and in this
way it becomes a topological \Star{}algebra.

Completing $\kcB$ with respect to $\|\ \|_2$ we obtain the Banach
*-algebra $\mathcal{HS}(\Bcal)$ of Hilbert-Schmidt operators of
$\Bcal$. The \emph{\cstar{}algebra of kernels} of $\Bcal$ is the
enveloping \cstar{}algebra of $\mathcal{HS}(\Bcal)$; it is denoted by
$\kB$. There is a canonical action of $G$ on $\kB$ given by the
formula $ \be_t(k)(r,s) = \Delta(t)k(rt,st)$ for $k\in \kcB$ and
$r,s,t\in G$.

The \cstar{}algebra $\Kb(L^2(\Bcal))$ of (generalised) compact operators of the Hilbert $B_e$\nb-module $L^2(\Bcal)$ can be canonically identified with an ideal in $\kB$: for $f,g\in \contc(\Bcal)$, the usual operator $\theta_{f,g}\in \Kb(L^2(\Bcal))$ given by $\theta_{f,g}(h)=f\braket{g}{h}$ is identified with the element $\lak f,g\rak\in\kcB$
defined by $\lak f,g\rak(r,s)=f(r)g(s)^*$. 
These elements span an ideal $I_c(\Bcal):=\spn\{\lak f,g\rak\colon f,g\in \contc(\Bcal)\}$ in $\kcB$. Its closure $I(\Bcal)$ is therefore a \cstar{}ideal of $\kB$.
The $\be$\nb-orbit of $I_c(\Bcal)$ is dense in $\kcB$ in the inductive limit topology.
Moreover, $I_c(\Bcal)$ is dense in $\kcB$ in the inductive limit topology if and only if $\Bcal$ is saturated.

\begin{remark}\label{rem:dual-coaction}
There is a canonical coaction $\delta_\B$ of $G$ on $C^*(\B)$, the so-called \emph{dual coaction}, and it is shown in \cite{Abadie:Enveloping} that $\kB$ is canonically isomorphic to the crossed product $C^*(\B)\rtimes_{\delta_\B}G$ by this coaction. Moreover, this isomorphism carries the canonical action of $G$ on $\kB$ to the dual action of $G$ on $C^*(\B)\rtimes_{\delta_\B}G$. Thus $\kB\cong C^*(\B)\rtimes_{\delta_\B}G$ as $G$-\cstar{}algebras. The dual coaction on $C^*(\B)$ is \emph{maximal} and its \emph{normalisation} is the dual coaction $\delta_\B^\red$ on $C^*_\red(\B)$ (see \cite{Buss-Echterhoff:Maximality}). This means that there exists a natural isomorphism
\begin{equation}\label{eq:crossed-product-kernels}
\kB\rtimes_\beta G\cong C^*(\B)\rtimes_{\delta_\B} G\rtimes_{\widehat{\delta}_\B}G\cong C^*(\B)\otimes \Kb(L^2(G))
\end{equation}
which factors through an isomorphism
$$\kB\rtimes_{\beta,\red} G\cong C^*_\red(\B)\rtimes_{\delta_\B^\red} G\rtimes_{\widehat{\delta}_\B^\red}G \cong C^*_\red(\B)\otimes \Kb(L^2(G)).$$
\end{remark}

Before we state our next result we introduce some notation.
We shall denote by $\Bcal_\be=\kB\times_\be G$ the semidirect product
Fell bundle associated to $\be$ (as defined in \cite[page 798]{Doran-Fell:Representations_2} for ordinary actions or, more generally, for twisted partial actions in \cite{Exel:TwistedPartialActions}).

Recall that $\Kb(L^2(\Bcal))$ can be identified with an ideal of $\kB$, so we have a (possibly non-faithful) representation of $\kB$ as adjointable operators of $L^2(\Bcal)$. We use the notation $Tf$ to represent the action of $T\in \kB$ on $f\in L^2(\Bcal)$. For every $k\in \kcB$ and $f\in \contc(\Bcal)$ we have $kf\in \contc(\Bcal)$ and
$kf(r)=\int_G k(r,s)f(s)\dd s$.

\begin{theorem}\label{thm:main}
  Let $\Bcal$ be a Fell bundle and denote by $\LtwoB$ its canonical $L^2$\nb-bundle, which is a full right Hilbert $\Bcal$\nb-bundle.  Then $\LtwoB$ is a full left Hilbert $\Bcal_\be$\nb-bundle with the action and inner product given by
  $$ T\delta_t f\delta_r =  \Delta(t)^{1/2} \be_{tr}^{-1}(T)f\delta_{tr}\qquad \labe f\delta_r,g\delta_s \rabe =\Delta(rs)^{-1/2} \be_r(\lak f,g\rak)\delta_{r\smu},$$
  where $T\delta_t\in \Bcal_\be$ and $f\delta_r,g\delta_s\in \LtwoB$.
  Moreover, the left and right Hilbert bundles structures of $\LtwoB$ are compatible and therefore $\LtwoB$ is a weak equivalence $\Bcal_\be-\Bcal$-bundle.
\end{theorem}
\begin{proof}
  The left action is clearly bilinear and the left inner product is linear in the first variable because the inner product $\lak\ ,\ \rak$ is linear in the first variable. Moreover,
  $$ \labe f\delta_r,g\delta_s \rabe^* = \Delta(r)^{-1/2}\Delta(s)^{-1/2} \be_{s\rmu}(\be_r(\lak f,g\rak^*))\delta_{\smu r} = \labe g\delta_s,f\delta_r \rabe.$$

  To show that the left operations are compatible, we compute
  \begin{align*}
    \labe T\delta_t f\delta_r,g\delta_s\rabe
      & = \Delta(rs)^{-1/2}\be_{tr}(\lak \be_{tr}^{-1}(T)f,g\rak)\delta_{tr\smu}.
  \end{align*}
and
  \begin{align*}
    T\delta_t \labe  f\delta_r,g\delta_s\rabe
      &  = \Delta(rs)^{-1/2} T\be_{tr}(\lak f,g\rak)\delta_{tr\smu}.
  \end{align*}
  Then the compatibility of the left operations will follow once we show that
  \begin{equation}\label{equation compatibility left operations}
       \be_{tr}(\lak \be_{tr}^{-1}(T)f,g\rak) = T\be_{tr}(\lak f,g\rak),
  \end{equation}
  for all $T\in \kB$ and $f,g\in L^2(\Bcal)$.
  But using linearity and continuity it suffices to consider $T\in \kcB$ and $f,g\in \contc(\Bcal)$.
  Then we can make all the computations in $\kcB$.
  With this assumption, the left hand side of (\ref{equation compatibility left operations}) evaluated at $(x,y)\in G^2$ is
  \begin{align*}
   \be_{tr}(\lak \be_{tr}^{-1}(T)h,g\rak)(x,y)
      & = \Delta(tr) [\be_{tr}^{-1}(T)f](xtr) g(ytr)^*\\
      & = \Delta(tr) \int_G \be_{tr}^{-1}(T)(xtr,z)f(z)\dd z\,  g(ytr)^* \\
      & = \int_G T(x,z\rmu\tmu)f(z)\dd z\,  g(ytr)^*.
  \end{align*}
  The right hand side of (\ref{equation compatibility left operations}) evaluated at $(x,y)\in G^2$ is
  \begin{align*}
     T\be_{tr}(\lak f,g\rak)(x,y)
      & = \int_G T(x,z) \Delta(tr)f(ztr)g(ytr)^*\dd z \\
      & = \Delta(tr) \int_G T(x,ztr \rmu\tmu) f(ztr)\dd z\, g(ytr)^*\\
      & = \int_G T(x,z \rmu\tmu) f(z)\dd z\, g(ytr)^*.
  \end{align*}
  Thus we have shown that (\ref{equation compatibility left operations}) holds and this implies that the left operations are compatible.

  Now note that $\labe f\delta_r,f\delta_r \rabe = \Delta(r)^{-1}\be_r(\lak f,f\rak)\delta_e\geq 0$ and
  $$\| \labe f\delta_r,f\delta_r \rabe \|=\Delta(r)^{-1}\|\be_r(\lak f,f\rak)\| = \Delta(r)^{-1}\|f \|_{L^2}^2= \|\lab f\delta_r,f\delta_r\rab\|, $$
  so the norms given by the left and right inner products agree.

  The continuity of the left operations follows directly from their definition and from the fact that the topologies of $\Bcal_\be=\kB\times G$ and $\LtwoB = L^2(\Bcal)\times G$ are the product topologies.

  Since the linear $\be$\nb-orbit of $\Kb(L^2(\Bcal))$ is dense in $\kB$, it follows from the definition of the left inner product that
  $$ \kB\delta_e = \spncl \{ \labe f\delta_r,g\delta_r \rabe \colon f,g\in L^2(\Bcal),\ r\in G \}. $$

  At this point we know that $\LtwoB$ is a full left Hilbert $\Bcal_\be$\nb-bundle. The proof will be completed once we show that the left and right operations are compatible, that is,
  \begin{equation*}
    \labe f\delta_r,g\delta_s\rabe h\delta_t = f\delta_r\lab g\delta_s,h\delta_t\rab.
  \end{equation*}
  It suffices to consider $f,g,h\in \contc(\Bcal)$ and by computing the left and right sides we obtain the following equivalent equation (without the place marker $\delta_{r\smu t}$):
  \begin{equation*}
    \Delta(s)^{-1}\be_{\tmu s}(\lak f,g\rak)h = f\int_G g(zs)^*h(zt)\dd z.
  \end{equation*}

  The left hand side evaluated at $x\in G$ is
  $$\Delta(t)^{-1} \int_G f(x\tmu s)g(z\tmu s)^*h(z)\dd z = f(x\tmu s) \int_G g(zs)^*h(zt)\dd z, $$
  which is exactly $f\int_G g(zs)^*h(zt)\dd z$ evaluated at $x$.
\end{proof}

The following result shows that every Fell bundle is strongly equivalent to the semidirect product Fell bundle of a partial action.

\begin{theorem}\label{theorem every Fell bundle is strongly equivalent to a semidirect product bundle}
  Let $\Bcal$ be a Fell bundle and denote by $\LtwoB=\{L_t\}_{t\in G}$ its canonical $L^2$\nb-bundle.
  If $\Xcal=\{L_t B_t^*B_t\}_{t\in G}$ and $\al$ is the restriction to $\Kb:=\Kb(L^2(\Bcal))$ of the canonical action on the \cstar{}algebra of kernels of $\Bcal$, then $\Xcal$ is a Banach subbundle of $\LtwoB$ and it is a strong equivalence $\Bcal_\al-\Bcal$\nb-bundle with the structure inherited from $\LtwoB$.
\end{theorem}

\begin{proof}
  Since $\{B_t^*B_t\}_{t\in G}$ is a continuous family of ideals of $B_e$, $\Xcal$ is a Banach subbundle of $\LtwoB$.

  To simplify our notation we define $\Kb_t:=\be_t(\Kb)\cap \Kb=\be_t(\Kb)\cdot\Kb$. Recall that the fiber over $t$ of $\Bcal_\al$ is $\Kb_t\delta_t$.

  To continue we identify the ideal of $\Kb$ corresponding to $B_t^*B_t$ through $L^2(\Bcal)$; we claim this ideal is $\Kb_\tmu$.
  Given $f,g\in \contc(\Bcal)$ and $a,b,c,d\in B_t$ we define $u,v\in \contc(\Bcal)$ by $u(r):=f(rt)a^*bd^*$ and $v(r):=g(rt)c^*$. Then $ \lak f a^*b ,g c^*d \rak = \be_\tmu(\lak u,v\rak)\in \Kb_\tmu$.
  This implies $L^2(\Bcal)B_t^*B_t\sbe \Kb_\tmu L^2(\Bcal)$.

  To prove $\Kb_\tmu L^2(\Bcal) \sbe L^2(\Bcal)B_t^*B_t$ it suffices to show
  $$\la \Kb_\tmu L^2(\Bcal), L^2(\Bcal) \ra_{B_e}\sbe B_t^*B_t.$$
  Note that $\Kb_\tmu L^2(\Bcal)=\be_\tmu(\Kb)\Kb L^2(\Bcal)= \be_\tmu(\Kb)L^2(\Bcal)$.
  If $k\in \kcB$ and $f,g\in \contc(\Bcal)$ then
  $$\la \be_\tmu(k)f,g\ra_{B_e}= \int_G \int_G f^*(s)k(r\tmu,s\tmu)^*g(r)\dd r\dd s.$$
  Now, if $k$ represents an element of the ideal $\Kb$, then $k(s,t)\in \Bcal_s\Bcal_t^*$ for all $s,t\in G$.
  Since $f^*(s)k(r\tmu,s\tmu)^*g(r)\in B_\smu B_{s\tmu} B_{t \rmu} B_r \sbe B_\tmu B_t=B_t^*B_t$ for all $r,s\in G$, we conclude that $\Kb_\tmu L^2(\Bcal)=L^2(\Bcal)B_t^*B_t$.

  The operations of $\Xcal$ are the ones inherited from $\LtwoB$.
  To show they are well defined and satisfy (1R-3R) and (1L-3L), it suffices to show
  \begin{equation}\label{equ:set inclusions}
   \Bcal_\al \Xcal\sbe\Xcal,\quad \Xcal\Bcal\sbe \Bcal\quad \mbox{and} \quad \labe \Xcal,\Xcal\rabe\sbe \Bcal_\al.
  \end{equation}

  To prove the first inclusion take $r,s\in G$, $T\in \Kb_r$ and $f\in L^2(\Bcal)B_s^*B_s$.
  By Cohen's factorisation theorem, we may decompose $f$ as $f=T'f'$ with $T'\in \Kb$ and $f'\in L^2(\Bcal)$.
  Since $L_{rs}B_{rs}^*B_{rs}= \be_{rs}^{-1}(\Kb)\Kb L^2(\Bcal)\delta_{rs}$ we have
  $$T\delta_r f\delta_s = \Delta(r)^{1/2}\be_{rs}^{-1}(T)f\delta_{rs} =\Delta(r)^{1/2}\be_{rs}^{-1}(T)T'f'\delta_{rs}\in L_{rs}B_{rs}^*B_{rs}. $$
  This shows that $\Bcal_\al \Xcal\sbe\Xcal$.

  The second inclusion in \eqref{equ:set inclusions} follows from $L_sB_s^*B_s B_r\sbe L_{sr}B_{sr}^*B_{sr}$, which we now show.
  Since $B_s^*B_s$ and $B_rB_r^*$ are ideals of $B_e$ and $B_r=B_rB_r^*B_r$, we have
  \begin{align*}
   L_sB_s^*B_s B_r
    & = L_sB_s^*B_s B_rB_r^*B_r
      = L_s B_rB_r^* B_s^*B_s B_r
      \sbe L_{sr} (B_sB_r)^*(B_sB_r)\\
    &  \sbe L_{sr} B_{sr}^*B_{sr}.
  \end{align*}

  The inclusion $\labe \Xcal,\Xcal\rabe\sbe \Bcal_\al$ follows from the fact that
  \begin{align*}
    \labe L_rB_r^*B_r,L_sB_s^*B_s\rabe
      & = \be_r(\lak \Kb_\rmu L^2(\Bcal),\Kb_\smu L^2(\Bcal) \rak )\delta_{r\smu}\\
      & = \be_r(\Kb_\rmu \cap \Kb_\smu )\delta_{r\smu}
        = \Kb_r \cap \Kb_{r\smu} \delta_{r\smu}\\
      & \sbe \Kb_{r\smu} \delta_{r\smu}.
  \end{align*}

  To finish we need to show $\X$ is strongly full on both sides.
  The strong fullness on the right follows from the computation
  \begin{align*}
   \spncl \lab  L_t B_t^*B_t,L_t B_t^*B_t\rab
      &= \spncl \lab  L^2(\Bcal)B_t^*B_t\delta_t,L^2(\Bcal)B_t^*B_t\delta_t\rab\\
      &=\spncl \{B_t^*B_t\la L^2(\Bcal),L^2(\Bcal)\ra_{B_e}B_t^*B_t \} = B_t^*B_t.
  \end{align*}
  And the strong fullness on the left is implied by
  \begin{align*}
   \spncl \labe  L_t B_t^*B_t,L_t B_t^*B_t\rabe
      &= \spncl \labe  L^2(\Bcal)B_t^*B_t\delta_t,L^2(\Bcal)B_t^*B_t\delta_t\rabe\\
      &=\spncl \{\be_t(\lak Tf,Tf\rak )\delta_e\colon T\in \Kb_\tmu, f\in L^2(\Bcal)\}\\
      & = \Kb_t\delta_e =\Kb_t\delta_t (\Kb_t\delta_t)^*.
  \end{align*}
\end{proof}

As an immediate consequence of the last result and Example
\ref{example globalization}, we get that every Fell bundle has a
Morita enveloping Fell bundle which is the semidirect product bundle
of an action on a \cstar{}algebra. We will see later (Corollary~\ref{cor:strong uniqueness}) that the Morita
enveloping Fell bundle is unique up to strong equivalence.

\section{Morita equivalence of actions and Fell bundles}

From the previous sections we know that the canonical action on the \cstar{}algebra of kernels of a Fell bundle determines the Morita equivalence class of that bundle. But what can we say about the canonical actions on the kernels of two Morita equivalent Fell bundles? Of course these actions, when viewed as Fell bundles are weakly equivalent. Our goal is to show that they are (Morita) equivalent as actions (hence strongly equivalent as Fell bundles).

Assume $\Acal$ and $\Bcal$ are Fell bundles over $G$ and let $\al$ and
$\be$ stand for the canonical actions on the \cstar{}algebras of
kernels of $\Acal$ and $\Bcal$, respectively.
If $\al$ is Morita equivalent to $\be$ (as actions on
\cstar{}algebras) then $\Bcal_\al$ and $\Bcal_\be$ are Morita
equivalent as Fell bundles (via a strong equivalence:
Example~\ref{ex:moritaeqpa}) and, by transitivity, $\Acal$ is strongly (hence weakly)
equivalent to $\Bcal$. Before proving the converse we prove the following.

\begin{lemma}\label{lem:subbundle-subalgebra}
  If $\Acal$ is a Fell subbundle of $\Bcal$ then the natural inclusion $\iota\colon \kcA\to \kcB$ extends
  to an injective (hence isometric) \Star{}homomorphism $\pi\colon \kb(\Acal)\to \kb(\Bcal)$.
\end{lemma}
\begin{proof}

First note that $\iota$ has a unique extension to a *-homomorphism $\mathcal{HS}(\Acal)\to \mathcal{HS}(\Bcal)$, which induces the *-homomorphism $\pi\colon \kb(\Acal)\to \kb(\Bcal)$ that we want to show is injective.
Since $\Acal$ is included in $\Bcal$, $L^2(\Acal)$ is contained
in $L^2(\Bcal)$ as a \cstar{}subtring and we may think of $I:=\mathbb{K}(L^2(\Acal))$ as a \cstar{}subalgebra of $\mathbb{K}(L^2(\Bcal))$
(see \cite[Proposition~4.1]{Abadie:Enveloping}).
This implies that $\pi|_{I}$ is injective.

Since $\pi$ is $\al-\be$-equivariant, the $\be$-closed linear orbit of $\pi(I)$ is $\pi(\kb(\Acal))$.
Moreover, $\pi|_I\colon I\to \pi(I)$ is an isomorphism of partial actions between $\al|_I$ and $\be|_{\pi(I)}$. Since $\al$ is an enveloping action of $\al|_I$ and $\be|_{\pi(\kb(\Acal))}$ one of $\be|_{\pi(I)}$,
the uniqueness of the enveloping action \cite[Theorem~2.1]{Abadie:Enveloping} implies the unique $\al- \be$-equivariant extension of $\pi|_I$is injective. But this extension is $\pi$.
\end{proof}

\begin{theorem}\label{theorem interpretation of morita equivalence in terms of Morita equivalence of actions}
  Let $\Acal$ and $\Bcal$ be Fell bundles over $G$ and let $\al$ and
  $\be$ be the canonical actions on the \cstar{}algebras of kernels of
  $\Acal$ and $\Bcal, $ respectively.
  Then $\Acal$ is weakly equivalent to $\Bcal$ if and only if $\al$ is Morita equivalent to $\be$.
\end{theorem}
\begin{proof}
  The converse follows by the comments we made at the beginning of the
  present section.
  For the direct implication assume that $\Xcal$ is an $\Acal-\Bcal$
  equivalence bundle.
  We may think of $\Acal$ and $\Bcal$ as full hereditary Fell
  subbundles of their linking Fell bundle $\Lb(\Xcal)$,
  so $\kA$ and $\kB$ are \cstar{}subalgebras
  of $\kb(\Lb(\Xcal))$ by Lemma~\ref{lem:subbundle-subalgebra}.
  Note that $\kcA \kb_c(\Lb(\Xcal)) \kcA\sbe \kcA$ and $\kcB \kb_c(\Lb(\Xcal)) \kcB\sbe \kcB$, so by taking completion we conclude that $\kA$ and $\kB$ are hereditary in $\kb(\Lb(\Xcal))$.

  To identify the equivalence bimodule inside $\kb(\Lb(\Xcal))$ implementing the Morita equivalence between $\kA$ and $\kB$ define
  $$\kb_c(\Xcal):=\{ f\in \kb_c(\Lb(\Xcal))\colon f(r,s)\in X_{r\smu}\mbox{ for all } r,s\in G\},$$
  and denote by $\kb(\Xcal)$ the completion of $\kb_c(\Xcal)$ in $\kb(\Lb(\Xcal))$.

  Note that $\kcA\kb_c(\Xcal)\sbe \kb_c(\Xcal)$ and $\kb_c(\Xcal)\kb_c(\Xcal)^*\sbe \kcA$, so by taking closures and linear span we obtain $\kA\kb(\Xcal)\sbe \kb(\Xcal)$ and $\kb(\Xcal)\kb(\Xcal)^*\sbe \kA$.
  By symmetry, we also have $\kb(\Xcal)\kB\sbe \kb(\Xcal)$ and $\kb(\Xcal)^*\kb(\Xcal)\sbe\kB$.

  The next step is to show that $\kb(\Xcal)^*\kb(\Xcal)=\kB$.
  Set $\Gamma:=\spncl \{ f^**g\colon f,g\in \kb_c(\Lb(\Xcal)) \}$, considered as a subset of the pullback $\Ccal$ of $\Bcal$ along the map $G\times G\to G,\ (r,s)\mapsto r\smu$.
  To show that $\Gamma$ is dense in the inductive limit topology on $\contc(\Ccal)=\kcB$ (and so dense in $\kB$) we use \cite[Lemma 5.1]{Abadie:Enveloping}.
  Now consider the algebraic tensor product $\contc(G)\odot \contc(G)$ as a dense subspace of $\contc(G\times G)$ for the inductive limit topology.
  Given $f,g\in \kcB$ and $\phi,\psi\in \contc(G)$ define $\phi\cdot f (r,s):=\phi(r)f(r,s)$.
  Then $(\phi\odot \psi)f^* *g = (\phi^* f)^* * (\psi g)\in \Gamma$.
  So $\contc(G)\odot \contc(G)\Gamma \sbe \Gamma$ and it suffices to show that, for every $(r,s)\in G$, $\{u(r,s)\colon u\in \Gamma\}$ is dense in $B_{r\smu}$.

  Given $t\in G$, $x\in X_{t\rmu}$ and $y\in X_{t\smu}$ take $f,g\in \contc(\Xcal)$ such that $f(t\rmu) = x$ and $g(t\smu)=y$.
  Consider the set of compact neighbourhoods of $e$, $\Ncal$, ordered by inclusion: $U\leq V$ if and only if $V\sbe U$ and for every $U\in \Ncal$ take $\phi\in \contc(G)^+$ with support contained in $U$ and such that $\int_G \phi_U^2(p)\dd p=1$
  Define $k_f^U\in \kb_c(\Xcal)$ as $k_f^U(p,q)=\phi(\tmu p)f(pq^{-1})$.
  Then
  \begin{align*}
   \lim_U (k^U_f)^* * k_g^U (r,s)
      & =\lim_U \int_G \phi_U(\tmu p)^2 \lab f(p\rmu),g(p\smu)\rab\dd p\\
      & = \lab f(t\rmu),g(t\smu)\rab =\lab x,y\rab.
  \end{align*}
  We conclude that the closure $C$ of $\{u(r,s)\colon u\in \Gamma\}$ contains $\lab X_{t\rmu},X_{t\smu}\rab$ for all $t\in G$.
  By Remark \ref{remark density of inner products} this implies $C=B_{r\smu}$.

  Now we know that $\kb(\Xcal)$ is a $\kA-\kB$\nb-equivalence bimodule.
  To finish the proof note that if $\gamma$ is the canonical action of
  $G$ on $\kb(\Lb(\Xcal))$ then $\kb(\Xcal)$ is $\gamma$\nb-invariant,
  $\al=(\gamma|_{\kb(\Xcal)})^l$ and $\be = (\gamma|_{\kb(\Xcal)})^r$
  (recall the notation from Example~\ref{ex:moritaeqpa}).
\end{proof}

\begin{corollary}(cf. \cite[Proposition~4.13]{Abadie-Ferraro:equivFB}).
If $\A$ and $\B$ are weakly equivalent Fell bundles over $G$, then their full and reduced cross-sectional \cstar{}algebras are (strongly) Morita equivalent. This equivalence respects the dual coactions. Conversely, if the dual coactions on the (full or reduced) cross-sectional \cstar{}algebras of $\A$ and $\B$ are equivalent (as coactions), then $\A$ and $\B$ are weakly equivalent as Fell bundles.
\end{corollary}
\begin{proof}
By Theorem~\ref{theorem interpretation of morita equivalence in terms of Morita equivalence of actions}, $\A$ and $\B$ are weakly equivalent if and only if their \cstar{}algebras of kernels $\kA$ and $\kB$ are $G$-equivariantly Morita equivalent. And by Remark~\ref{rem:dual-coaction}, we have canonical isomorphisms of $G$-\cstar{}algebras $\kA\cong C^*(\A)\rtimes_{\delta_\A}G\cong C^*_\red(\A)\rtimes_{\delta_\A^\red}G$ and $\kB\cong C^*(\B)\rtimes_{\delta_\B}G\cong C^*_\red(\B)\rtimes_{\delta_\B^\red}G$, where $\delta_\A^{(\red)}$ and $\delta_\B^{(\red)}$ denote the dual coactions on $C^*_{(\red)}(\A)$ and $C^*_{(\red)}(\B)$, respectively. The coaction $\delta_\A$ is maximal, and the coaction $\delta_\A^{\red}$ is normal (the normalisation of $\delta_\A$), see \cite{Buss-Echterhoff:Maximality}. This means that $\delta_\A$ (resp. $\delta_\A^\red$) is Morita equivalent to the dual coaction on $\kA\rtimes_\alpha G$ (resp. $\kA\rtimes_{\alpha,\red} G$), and a similar assertion holds for $\B$ in place of $\A$. Combining all this and the standard result that equivalent actions or coactions have equivalent (full or reduced) crossed products, the desired result now follows.
\end{proof}

\begin{remark}
The above result extends to the exotic cross-sectional \cstar{}algebras $C^*_\mu(\B)$ associated to Morita compatible cross-product functors $\rtimes_\mu$ as defined in \cite{Buss-Echterhoff:Maximality}. This is because, by definition, $C^*_\mu(\B)$ is the quotient of $C^*(\B)$ that turns the isomorphism~\eqref{eq:crossed-product-kernels} into an isomorphism
$$\kB\rtimes_{\be,\mu}G\cong C^*_\mu(\B)\otimes \Kb(L^2(G)).$$
And this isomorphism preserves the dual coactions whenever the exotic
crossed-product admits such a coaction; this means that $\rtimes_\mu$
is a duality cross-product functor in the language of
\cite{Buss-Echterhoff-Willett:Exotic-Survey} (see also \cite{Abadie-Ferraro:equivFB}). This is a big class of functors and, in particular, includes all correspondence functors (see \cite{Buss-Echterhoff:Maximality}*{Corollary~4.6}).
\end{remark}

Recall that a Fell bundle $\A$ is \emph{amenable} if the regular representation $\lambda_\A\colon C^*(\A)\to C^*_\red(\A)$ is faithful.

\begin{corollary}
Let $\A$ and $\B$ be two weakly equivalent Fell bundles. Then $\A$ is amenable if and only if $\B$ is amenable.
\end{corollary}
\begin{proof}
As already explained in Remark~\ref{rem:dual-coaction}, the canonical isomorphism $C^*(\A)\otimes \Kb(L^2(G))\cong \kA\rtimes_\alpha G$ factors (via the regular representations) through an isomorphism $C^*_\red(\A)\otimes \Kb(L^2(G))\cong \kA\rtimes_{\alpha,\red} G$. This means that $\A$ is amenable if and only if the $G$-action $\alpha$ on $\kA$ is amenable in the sense that the regular representation
$\kA\rtimes_{\alpha} G\to \kA\rtimes_{\alpha,\red} G$ is faithful. Since $\A$ is weakly equivalent to $\B$ if and only if the $G$-actions on $\kA$ and $\kB$ are equivalent by Theorem~\ref{theorem interpretation of morita equivalence in terms of Morita equivalence of actions}, the assertion follows from the well-known result that amenability of actions is preserved by equivalence of actions.
\end{proof}

We will have more to say about amenability in Section~\ref{sec:Fell-bundles-amenability}.

\begin{corollary}
  Let $\al$ and $\be$ be partial actions of $G$ on \cstar{}algebras.
  Then $\Bcal_\al$ is weakly equivalent to $\Bcal_\be$ if and only if $\al$ and $\be$ have equivalent Morita enveloping actions.
\end{corollary}
\begin{proof}
  Recall from \cite{Abadie:Enveloping} that the canonical action on the \cstar{}algebra of kernels of $\Bcal_\al$ is a Morita enveloping action for $\al$ and that Morita enveloping actions are unique up to Morita equivalence of actions.
  Then the proof follows directly using transitivity of Morita equivalence of Fell bundles and the last theorem.
\end{proof}

Our next theorem will show that strong equivalence of Fell bundles
corresponds to Morita equivalence of partial actions in the ordinary
sense (recall Example~\ref{ex:moritaeqpa} and the notation used
there). First we need the following auxiliary result.

\begin{lemma}
  Suppose $X$ is an $A-B$\nb-equivalence bimodule, $\gamma$ is an action of $G$ on $X$ and that $I$ and $J$ are \cstar{}ideals of $A$ and $B$ (respectively) such that $IX=XJ$.
  Then $\gamma^l|_I=(\gamma|_{IX})^l$ and $\gamma^r|_J=(\gamma|_{XJ})^r$.
  In particular $\gamma^l|_I$ is Morita equivalent to $\gamma^r|_J$ (as partial actions).
\end{lemma}
\begin{proof}
  To simplify the notation, we denote $\al:=\gamma^l$ and $\be=\gamma^r$.
  Since for all $t\in G$ we have $\al_t(I)X = \gamma_t(IX)$, the ideal of $X$ corresponding to $I_t=I\cap \al_t(I)$ is $IX\cap \gamma_t(IX)$.
  By symmetry we obtain $I_t X=XJ_t$.
  Since $I_\tmu=\spncl {}_A\la X\cap \gamma_t(IX),X\cap \gamma_t(IX)\ra$ and for $x,y\in X\cap \gamma_t(IX)$ we have
  $$\al_t({}_A\la x,y\ra)={}_A\la \gamma_t(x),\gamma_t(y)\ra = (\gamma|_{IX})^l({}_A\la x,y\ra).$$
  We conclude that $\al|_I = (\gamma|_{IX})^l$.
  The rest follows by symmetry.
\end{proof}

\begin{theorem}\label{the:characterization-strong-equivalence}
  Let $\Acal$ and $\Bcal$ be Fell bundles over $G$ and denote by $\al$ and $\be$ the restrictions of the canonical action on the \cstar{}algebras of kernels, $\kA$ and $\kB$, to $\Kb_\Acal:=\Kb(L^2(\Acal))$ and $\Kb_\Bcal:=\Kb(L^2(\Bcal))$, respectively.
  Then $\Acal$ is strongly equivalent to $\Bcal$ if and only if $\al$ is Morita equivalent to $\be$.
\end{theorem}
\begin{proof}
 If $\al$ is equivalent to $\be$ in the usual sense (as defined in
 \cite{Abadie:Enveloping}), then their associated Fell bundles
 $\Bcal_\al$ and $\Bcal_\be$ are strongly equivalent, as shown in Example~\ref{ex:moritaeqpa}. Since $\A$
 (resp. $\B$) is strongly equivalent to $\B_\al$ (resp. $\B_\be$) by
 Theorem~\ref{theorem every Fell bundle is strongly equivalent to a
   semidirect product bundle}, the strong equivalence between $\A$ and $\B$ follows by transitivity (Theorem \ref{theorem:strong equivalence is an equivalence relation}).

 For the converse assume that $\Xcal$ is a strong equivalence $\Acal-\Bcal$\nb-bundle.
 Let $\kb(\Xcal)$ be the $\kA-\kB$\nb-equivalence bimodule constructed in the proof of Theorem \ref{theorem interpretation of morita equivalence in terms of Morita equivalence of actions}, which was constructed inside $\kb(\Lb(\Xcal))$.

 By the previous Lemma and the proof of Theorem \ref{theorem interpretation of morita equivalence in terms of Morita equivalence of actions} it suffices to show that $\Kb_\Acal \kb(\Xcal)=\kb(\Xcal)\Kb_\Bcal$.
 Let $\Ccal$ be the pullback of $\Bcal$ along the map $\rho\colon G^2\to G,\ (r,s)\mapsto r\smu$.
 We think of $\kcB$ as $\contc(\Ccal)$.

 Note that the bundle $\Dcal^\Bcal:=\{B_rB_{\smu}\}_{(r,s)\in G^2}$ is a Banach subbundle of $\Ccal$ (recall Notation \ref{notation of product of sets}). Moreover, using \cite[Lemma 5.1]{Abadie:Enveloping} (as in the proof of Theorem \ref{theorem interpretation of morita equivalence in terms of Morita equivalence of actions}) one shows that $\spn\{ \lak f,g\rak\colon f,g\in \contc(\Bcal) \}$ is dense in $\contc(\Dcal^\Bcal)$ for the inductive limit topology. Thus $\Kb_\Bcal$ is the closure of $\contc(\Dcal^\Bcal)$ in $\kB$.

 In a similar way define $\Dcal^{\Xcal\Bcal}:=\{X_rB_\smu\}_{(r,s)\in G^2}$, which is a Banach subbundle of the pullback of $\Lb(\Xcal)$ along $\rho$.
 Note that $\kb_c(\Xcal)\spn \lak \contc(\Bcal),\contc(\Bcal)\rak \sbe \contc(\Dcal^\Xcal)$ and that, by arguments similar to those in the proof of Theorem \ref{theorem interpretation of morita equivalence in terms of Morita equivalence of actions},  $\kb_c(\Xcal)\spn \lak \contc(\Bcal),\contc(\Bcal)\rak$ is dense in $\contc(\Dcal^{\Xcal\Bcal})$ for the inductive limit topology.
 Thus $\kb(\Xcal)\kB$ is the closure of $\contc(\Dcal^{\Xcal\Bcal})$ in $\kb(\Lb(\Xcal))$.

 By symmetry we obtain that $\kA\kb(\Xcal)$ is the closure in $\kb(\Lb(\Xcal))$ of the space of compactly supported continuous sections of the bundle $\Dcal^{\Acal\Xcal}:=\{ A_rX_\smu \}_{(r,s)\in G^2}$.
 Thus all we need to show is that $\Dcal^{\Acal\Xcal}=\Dcal^{\Xcal\Bcal}$.
 It suffices to prove that $A_rX_s=X_rB_s$ for all $(r,s)\in G^2$.
 If we think of $B_s$ as a right full $B_s^*B_s$\nb-Hilbert module and we use that $B_s^*B_s=\spncl \lab X_s,X_s\rab$, then we conclude that
 $$ X_rB_s =\spncl X_r B_s \lab X_s,X_s\rab = \spncl \laa X_r,X_sB_\smu\raa X_s \sbe A_rX_s.$$
 Conversely,
\begin{equation*}
A_rX_s = \spncl \laa X_r,X_r\raa A_r X_s = \spncl X_r\laa X_r,A_rX_s\raa \sbe X_rB_s. \qedhere
\end{equation*}
\end{proof}

The next result shows that our notion of strong equivalence of Fell bundles recovers exactly the notion of equivalence of partial actions (introduced in \cite{Abadie:Enveloping}).

\begin{corollary}
  Let $\al$ and $\be$ be partial actions of $G$ on \cstar{}algebras.
  Then $\al$ is Morita equivalent to $\be$ if and only if $\Bcal_\al$ is strongly equivalent to $\Bcal_\be$.
\end{corollary}
\begin{proof}
  This follows from Theorem~\ref{the:characterization-strong-equivalence} together with the fact (proved in \cite{Abadie:Enveloping}) that the restricted partial action to $\Kb(L^2(\Bcal_\al))$ of the canonical global action on $\kb(\Bcal_\al)$ is Morita equivalent to~$\al$.
\end{proof}

\begin{corollary}\label{cor:sats}
Two saturated Fell bundles are weakly equivalent if and only they are strongly equivalent.
\end{corollary}
\begin{proof}
If $\A$ is a saturated Fell bundle, then its \cstar{}algebra of kernels $\kA$ coincides with $\Kb(L^2(\A))$ and the canonical partial action is already global. The result now follows as a direct combination of Theorems~\ref{theorem interpretation of morita equivalence in terms of Morita equivalence of actions} and~\ref{the:characterization-strong-equivalence}.
\end{proof}

Now, combining Corollary~\ref{cor:sats},
Example~\ref{example globalization}, Remark~\ref{rem:weak uniqueness},
and Theorem~\ref{theorem every Fell bundle is strongly equivalent to a semidirect product bundle}, we get:
\begin{corollary}\label{cor:strong uniqueness}
Every Fell bundle has a Morita enveloping Fell bundle which is the
semidirect product bundle of an action on a C*-algebra. This action is unique up to (strong) Morita equivalence of actions on \cstar{}algebras.
Any two Morita enveloping Fell bundles of a Fell bundle are strongly equivalent.
\end{corollary}

\section{Partial actions associated with Fell bundles.}
\label{sec:partial actions}

By the \emph{spectrum} of a \cstar{}algebra $A$ we mean the space $\hat{A}$
of unitary equivalence classes $[\pi]$ of irreducible representations $\pi$ of $A$ with the Jacobson topology induced from the primitive ideal space $\Prim(A)=\{\ker(\pi): [\pi]\in \hat{A}\}$.
Open subsets of $\hat{A}$ or $\Prim(A)$ correspond bijectively to ideals of $A$: the open subset of $\Prim(A)$ (resp. $\hat{A}$) associated with an ideal $I\sbe A$ is $\{p\in \Prim(A): I\varsubsetneq p\}$ (resp. $\{[\pi]: \pi|_I\not=0\}$).

As shown in \cite{Abadie-Abadie:Ideals}, every Fell bundle
$\Bcal=(B_t)_{t\in G}$ over a discrete group $G$ induces a partial action
$\hat{\alpha}^\Bcal$ of $G$ on the spectrum $\hat{B}_e$ of $B_e$. We briefly
recall how $\hat{\alpha}^\Bcal$ is defined. For each $t\in G$ we let
$D^\Bcal_t:=B_tB_t^*$; then $D^\Bcal_t$ is an ideal of $B_e$ and $B_t$ can be viewed as a
$D^\Bcal_t-D^\Bcal_{t^{-1}}$ imprimitivity bimodule. Let
$\mathcal{V}_t^{\Bcal}:=\{[\pi]\in \hat{B}_e:\pi|_{D^\Bcal_t}\neq 0 \}$ be the open subset of $\hat{B_e}$ associated with $D^\Bcal_t$. If
$[\pi]\in \mathcal{V}_{t^{-1}}^{\Bcal}$, then
$[\pi|_{D^\Bcal_{t^{-1}}}]\in\hat{D}^\Bcal_{t^{-1}}$, and therefore
$[\textrm{Ind}_{B_t}(\pi|_{D^\Bcal_{t^{-1}}})]\in \hat{D}^\Bcal_{t}$. Then
$\hat{\alpha}_t([\pi])\in\mathcal{V}_t^\Bcal$ is defined to be the
class of the unique
extension of $\textrm{Ind}_{B_t}(\pi|_{D^\Bcal_{t^{-1}}})$ to all of $B_e$
(see the second statement of \ref{lem:rest-ext}). In fact one can give
a more direct definition:
$\hat{\alpha}_t([\pi])=[\textrm{Ind}_{B_t}\pi]$ (this is a consequence
of our Lemma~\ref{lem:rest-ext}(3)). Here and throughout, if $E$ is a Hilbert $C'-C$-bimodule for \cstar{}algebras $C',C$, we write $\Ind_E(\pi)$ for the representation of $C'$ induced via $E$ from a representation $\pi\colon C\to \Bb(H)$. Recall that $\Ind_E(\pi)$ acts on the (balanced tensor product) Hilbert space $E\otimes_\pi H$ by the formula $\Ind_E\pi(x)(y\otimes_\pi h):=x\cdot y\otimes_\pi h$ for all $x\in C'$, $y\in E$ and $h\in H$.

Only discrete groups are considered in \cite{Abadie-Abadie:Ideals}. But in this section
we prove that the partial action $\hat{\alpha}^\Bcal$ is
always continuous if $G$ is a locally compact group. We also show that strongly equivalent Fell bundles have
isomorphic partial actions, and that the action of a saturated Fell bundle is the
enveloping action of the partial action of any Fell bundle weakly
equivalent with the former.

We begin with some preliminary results about induced representations via Hilbert bimodules; most of them are certainly well known, but we include the proofs here for convenience. Let $A$ be a
\cstar{}subalgebra of the \cstar{}algebra $C$, and
suppose $\pi:C\to \Bb(H)$ is a nondegenerate representation of $C$. We
denote by $\pi_A$ the nondegenerate part of the restriction $\pi|_A$,
that is, $\pi_A:A\to \Bb(H_A)$ is given by $\pi_A(a)h:=\pi(a)h$ for all
$a\in A$ and $h\in H_A$, where $H_A:=\overline{\textrm{span}}\{\pi(a)h:a\in
A,h\in H\}$, the essential space of $\pi|_A$.

\begin{lemma}\label{lem:first}
Let $E$ be a Hilbert $C'-C$-bimodule, $A'$ and $A$ \cstar{}subalgebras of
$C'$ and $C$ respectively, and $F\subseteq E$ such that $F$ is a
Hilbert $A'-A$-bimodule with the structure inherited from $E$. Suppose
$\pi:C\to\Bb(H)$ is a representation and $K\sbe H$
is a closed subspace which is invariant under $\pi_A$. Then:
\begin{enumerate}
\item There exists a unique isometry $V:F\otimes_{\pi_A}K\to
  E\otimes_\pi H$ that satisfies
  $V(x\otimes_{\pi_A}k)=x\otimes_{\pi}k$ for all $x\in F$, $k\in K$.
\item The isometry $V$ intertwines $(\textrm{Ind}_E\pi)_{A'}$ and
  $\textrm{Ind}_F\pi_A$, that is,
  \[\textrm{Ind}_E\pi(a')V=V\textrm{Ind}_F\pi_A(a')\quad \mbox{for all }a'\in A'.\]
\end{enumerate}
\end{lemma}
\begin{proof}
For a finite sum of elementary tensors $\sum_ix_i\otimes_{\pi_A}k_i\in F\otimes_{\pi_A}K$ we compute:
\begin{gather*}\left\|\sum_ix_i\otimes_{\pi_A}k_i\right\|^2
=\sum_{i,j}\langle
x_i\otimes_{\pi_A}k_i,x_j\otimes_{\pi_A}k_j\rangle_K
=\sum_{i,j}\langle k_i,\pi_A(\langle x_i,x_j\rangle_C)
k_j\rangle_K\\
=\sum_{i,j}\langle k_i,\pi(\langle x_i,x_j\rangle_C)
k_j\rangle_H
=\left\|\sum_ix_i\otimes_{\pi}k_i\right\|^2
\end{gather*}
Then there exists an isometry $V:F\otimes_{\pi_A}K\to
E\otimes_{\pi}H$ such that $V(x\otimes_{\pi_A}k)=x\otimes_{\pi}k$ for all
$x\in F$, $k\in K$, as claimed in (1).
Now, if $a'\in A'$, $x\in F$ and $k\in K$:
\[
\textrm{Ind}_E\pi(a')V(x\otimes_{\pi_A}k)
=a'x\otimes_{\pi}k
=V(a'x\otimes_{\pi_A}k)
=V\textrm{Ind}_F\pi_A(a')(x\otimes_{\pi_A}k),
\]
which proves (2).
\end{proof}

\begin{lemma}\label{lem:rest-ext}
Let $\pi:C\to \Bb(H)$ be a nondegenerate representation of a \cstar{}algebra $C$.
\begin{enumerate}
 \item Let $Y$ be a closed right ideal of $C$ and
   $A:=YY^*$ (recall Notation~\ref{notation of product of sets}),
   which is a hereditary
   \cstar{}subalgebra of $C$. Consider $Y$ as a Hilbert
   $A-C$-bimodule. Then $\textrm{Ind}_{Y}\pi$ is equivalent
   to~$\pi_A$.
 \item Let $I$ be a closed two-sided ideal of $C$ and
   let $F_I$ be $I$ with its natural structure of Hilbert
   $C-I$-bimodule. If $\rho:I\to\Bb(K)$ is a nondegenerate
   representation, let $\tilde{\rho}:C\to\Bb(K)$ be the unique
   extension of $\rho$ to a representation of $C$ on $K$, which is
   determined by $\tilde{\rho}(c)(\rho(x)k)=\rho(cx)k$ for all $c\in
   C$, $x\in I$ and $k\in K$. Then $\textrm{Ind}_{F_I}\rho$ is
   equivalent to $\tilde{\rho}$.
 \item Suppose $\pi$ is irreducible. If $C'$ is a \cstar{}algebra and
   $E$ is a Hilbert $C'-C$-bimodule such that $\pi$ does not vanish on the ideal $J:=\spncl\langle
     E,E\rangle_C$, then $\textrm{Ind}_E\pi$ is irreducible and
   equivalent to $\textrm{Ind}_{F_I}(\textrm{Ind}_E\pi_J)$, where $I:=\spncl_{C'}\langle E,E\rangle$.
\end{enumerate}
\end{lemma}
\begin{proof}
Let $K':=\overline{\textrm{span}}\{\pi(y)h:\,y\in Y,\, h\in H\}$. Note
that $K'$ agrees with the essential space $H_A$ of $\pi|_A$. In fact
we have $H_A\subseteq K'$ because $A\subseteq Y$, and the
reverse inclusion follows from the fact that $y=\lim_\lambda e_\lambda y$ and hence
$\pi(y)=\lim_\lambda \pi(e_\lambda)\pi(y)$ for every $y\in Y$ and
every approximate unit $(e_\lambda)$
of $A$. Now, if $\sum_iy_i\otimes h_i$ is a finite sum of elementary tensors in $Y\otimes_\pi H$, then
\[\left\| \sum_iy_i\otimes h_i\right\|^2
=\sum_{i,j}\langle y_i\otimes h_i,y_j\otimes h_j\rangle
=\sum_{i,j}\langle h_i,\pi(y_i^*y_j)h_j\rangle_H
=\left\|\sum_i\pi(y_i)h_i\right\|^2.
\]
Thus we have an isometry $U:Y\otimes_\pi H\to H_A$ such that $y\otimes
h\mapsto \pi(y)h$. This isometry is surjective by the previous observation, hence $U$ is a unitary. Finally, if $a\in A$,
$y\otimes h\in Y\otimes_\pi H$:
\[
U\textrm{Ind}_{Y}\pi (a)(y\otimes h)
=U(ay\otimes h)
=\pi(ay)h
=\pi(a)\pi(y)h
=\pi(a)U(y\otimes h),
\]
which proves our first statement.

To prove (2) we observe that exactly the same argument used in
the proof of (1) shows that there is a unitary operator
$U:F_I\otimes_\rho K\to
K$ that intertwines $\textrm{Ind}_{F_I}\rho$ with $\tilde{\rho}$.

As for (3), since $I$ and $J$ are the ideals generated by the left and right inner products of the bimodule $E$, we may view $E$ as an
$I-J$-imprimitivity bimodule. By (1), $\pi_J$ is
irreducible. And the essential space of $\pi_J$ is $H$ because $\pi(J)H$ is a non-zero $\pi$-invariant subspace of $H$, and
$\pi$ is irreducible. Since $E$ is an $I-J$-imprimitivity
bimodule, $\textrm{Ind}_E{\pi_J}:I\to
\Bb(E\otimes_{\pi_J}H)$ also is irreducible. Now let $V:E\otimes_{\pi_J}H\to
E\otimes_{\pi}H$ be the isometry provided by (1) of \ref{lem:first},
which in this case is obviously surjective, thus a unitary
operator. Since, according to Lemma~\ref{lem:first}(2), $V$ intertwines
$(\textrm{Ind}_E\pi)_I$ and $\textrm{Ind}_E\pi_J$, and the latter is
irreducible, then so is $(\textrm{Ind}_E\pi)_I$. Therefore
$\textrm{Ind}_E\pi$ is irreducible. Moreover, if
$\rho:=\textrm{Ind}_E\pi_J$, it is an easy task to show that
$V\tilde{\rho}(c')=V\textrm{Ind}_E\pi(c')$ for all $c'\in C'$. This
ends the proof, for $\textrm{Ind}_{F_I}(\textrm{Ind}_E\pi_J)$ and
$\tilde{\rho}$ are equivalent by (2).
\end{proof}
\begin{lemma}\label{lem:third}
Suppose, in the conditions of Lemma~\ref{lem:first}, that $\pi$ is
irreducible, $A'$ is a hereditary \cstar{}subalgebra
  of $C'$, and $\pi|_{\langle E,E\rangle_C}\neq 0$. Then the isometry $V$
  is a unitary operator,
  and $[(\textrm{Ind}_E\pi)_{A'}]=[\textrm{Ind}_F\pi_A]$.
\end{lemma}
\begin{proof}
Since $\pi$ is irreducible, then so is
$\textrm{Ind}_E\pi$ by Lemma~\ref{lem:rest-ext}(3). Moreover, if $A'$ is
a hereditary \cstar{}subalgebra of $C'$, then
$(\textrm{Ind}_E\pi)_{A'}$ is either zero or an
irreducible representation of $A'$ (\cite[Theorem~5.5.2]{Murphy:book}). But
$(\textrm{Ind}_E\pi)_{A'}$ cannot
be zero because of Lemma~\ref{lem:first}(2) and the fact that $\pi_A$,
and therefore
$\textrm{Ind}_F\pi_A$, are non-zero representations. Now it follows
from Lemma~\ref{lem:first}(2) that $V(F\otimes_{\pi_A}K)$ is a non-zero
$(\textrm{Ind}_E\pi)_{A'}$-invariant subspace of $E\otimes_{\pi}H$ and,
since $(\textrm{Ind}_E\pi)_{A'}$ is irreducible, we must have
$V(F\otimes_{\pi_A}K)=E\otimes_{\pi}H$. That is, $V$ is a surjective
isometry, which ends the proof.
\end{proof}

We show next that two strongly equivalent Fell bundles give rise
to isomorphic partial actions on spectra level.

\begin{theorem}\label{thm:eq partial actions}
Let $\Acal$ and $\Bcal$ be Fell bundles over a discrete group $G$,
and suppose that $\Xcal$ is a strong $\Acal-\Bcal$ equivalence. Let
$\mathsf{h}_{X_e}:\hat{B}_e\to\hat{A}_e$ be the Rieffel homeomorphism
associated to
the $A_e-B_e$ imprimitivity bimodule $X_e$. Then
$\mathsf{h}_{X_e}:\hat{\alpha}^\Bcal\to \hat{\alpha}^\Acal$ is an
isomorphism of partial actions.
\end{theorem}
\begin{proof}
Let $\Ccal=(C_t)_{t\in G}$ stand for the linking bundle of
$\Xcal$. Then $C_e=\mathbb{L}(X_e)$ (the linking algebra of $X_e$),
and
$Y=\left(\begin{smallmatrix} A_e & X_e\\ 0 &
    0 \end{smallmatrix}\right)$ is an $A_e-C_e$ imprimitivity
bimodule, so it defines the Rieffel homeomorphism $\mathsf{h}_Y\colon
\hat{C}_e\to \hat{A}_e$, which in our case is given by
$\mathsf{h}_Y([\pi])=[\pi_{A_e}]$ (\ref{lem:rest-ext} (1)\,).
Therefore the Rieffel correspondence $\mathsf{R}\colon \mathcal{I}(C_e)\to
\mathcal{I}(A_e)$ between the ideals of $C_e$ and $A_e$, defined by
$Y\cong A_e\oplus X_e$, is given by $\mathsf{R}(I)=A_e\cap I$.

We claim that $\mathsf{h}_Y$ is an isomorphism between
$\hat{\al}^{\Ccal}$ and $\hat{\al}^{\Acal}$. First note that
$\mathsf{R}(D^\Ccal_t) = D^\Acal_t$, that is: $D^\Ccal_t\cap
A_e=C_tC_t^*\cap A_e=A_tA_t^*=D^\Acal_t$. In fact, a
simple computation gives
\[C_tC_t^*
=\begin{pmatrix} A_t^*A_t+\langle
    X_{t^{-1}},X_{t^{-1}}\rangle_\Acal & A_t^*X_t+X_{t^{-1}}B_t\\
    \widetilde{(A_t^*X_t+X_{t^{-1}}B_t)} &
    \langle
    X_{t},X_{t}\rangle_\Bcal+B_t^*B_t \end{pmatrix}.
\]
Thus we have $\mathsf{h}_Y(\mathcal{V}^\Ccal_t)=\mathcal{V}^\Acal_t$ for all
$t\in G$. Now take $[\pi]\in \mathcal{V}_{t^{-1}}$. We must
show $\mathsf{h}_Y(\alpha_t^{\Ccal}([\pi])) =
  \al_t^\Acal(\mathsf{h}_Y[\pi])$, that is:
  $[\textrm{Ind}_Y\textrm{Ind}_{C_t}\pi]=[\textrm{Ind}_{A_t}\textrm{Ind}_Y\pi]$.
  Since for every representation $\rho$ of $C_e$ we have
  $[\textrm{Ind}_Y\rho]=[\rho_{A_e}]$, it is enough to show that
  $[(\textrm{Ind}_{C_t}\pi)_{A_e}]=[\textrm{Ind}_{A_t}\pi_{A_e}]$.
 But this follows at once from
  Lemma~\ref{lem:third} by taking $C=C'=C_e$, $A=A'=A_e$, $E=C_t$ and
  $F=A_t$.

Similarly, if we now consider the $C_e-B_e$ imprimitivity
bimodule $Z:=\left(\begin{smallmatrix} 0 & X_e\\ 0 &
    B_e \end{smallmatrix}\right)$, then the Rieffel homeomorphism
$\mathsf{h}_Z:\hat{B}_e\to\hat{C}_e$ is also an isomorphism
$\mathsf{h}_Z:\hat{\alpha}^\Bcal\to \hat{\alpha}^\Ccal$. Then
$\mathsf{h}_Y\circ\mathsf{h}_Z:\hat{\alpha}^\Bcal\to
\hat{\alpha}^\Acal$ is an isomorphism of partial
actions. Besides, since $Y\otimes_{C_e}Z=\left(\begin{smallmatrix} A_e
    & X_e\\ 0 & 0 \end{smallmatrix}\right)\otimes
_{C_e}\left(\begin{smallmatrix} 0 & X_e\\ 0 &
    B_e \end{smallmatrix}\right) = X_e$, and
$\mathsf{h}_Y\circ\mathsf{h}_Z=\mathsf{h}_{\,Y\otimes_{C_e}Z}$, we
conclude that $\mathsf{h}_{X_e}:\hat{\alpha}^\Bcal\to
\hat{\alpha}^\Acal$ is an isomorphism of partial
actions.
\end{proof}

Given a locally compact Hausdorff group $G$, let $G_d$ denote the
group $G$ when considered with the discrete topology. Similarly, if
$\Bcal$ is a Fell bundle over $G$, let $\Bcal_d$ be the Fell bundle
over $G_d$ instead of $G$.
\begin{proposition}\label{prop:contpa}
Let $\Bcal$ be a Fell bundle over $G$, and let $G_d$ and $\Bcal_d$ as
above. Then the partial action $\hat{\alpha}^{\Bcal_d}$ of $G_d$ on
$\hat{B}_e$ is a continuous partial action of $G$ on $\hat{B}_e$.
\end{proposition}
\begin{proof}
Consider the canonical action $\beta$ of $G$ on $\kb(\Bcal)$, and let
$\gamma$ be its restriction to $\Kb(L^2(\Bcal))$. Let
$\Acal:=\Bcal_{\gamma}$. Since, by Theorem~\ref{theorem every Fell
  bundle is strongly equivalent to a semidirect product bundle},
$\Acal$ is strongly equivalent to $\Bcal$, then also $\Acal_d$ is
strongly equivalent to $\Bcal_d$. Note that if we forget the topology
of $G$, then $\hat{\gamma}=\hat{\alpha}^{\Acal_d}$, so the latter is a
continuous partial action of $G$ on the spectrum of $\Kb(L^2(\Bcal))$. On
the other hand $\hat{\alpha}^{\Acal_d}$ and $\hat{\alpha}^{\Bcal_d}$
are isomorphic partial actions by Theorem~\ref{thm:eq partial
  actions} and, since $\hat{\alpha}^{\Acal_d}$ is continuous, so must
be $\hat{\alpha}^{\Bcal_d}$.
\end{proof}

The previous result allows us to associate a partial action to
every Fell bundle, not only to those over discrete groups:

\begin{definition}
Let $\Bcal$ be a Fell bundle over $G$, and denote by
$\hat{\alpha}^{\Bcal}$ the partial action $\hat{\alpha}^{\Bcal_d}$
considered as a partial action of $G$ on $\hat{B}_e$. We say that
$\hat{\alpha}^{\Bcal}$ is the partial action associated to $\Bcal$.
\end{definition}

Now Theorem~\ref{thm:eq partial actions} can be stated for
Fell bundles over arbitrary groups:

\begin{corollary}
  Suppose $\Xcal$ is a strong $\Acal-\Bcal$-equivalence bundle.
  If $\mathsf{h}\colon \hat{B}_e\to \hat{A}_e$ is the Rieffel
  homeomorphism induced by the $A_e-B_e$-equivalence bimodule $X_e$,
  then $\mathsf{h}$ is an isomorphism between $\hat{\al}^{\Acal}$ and
  $\hat{\alpha}^{\Bcal}$.
\end{corollary}

\begin{corollary}
Let $\Bcal$ be a Fell bundle, and $\beta:G\times \kb(\Bcal)\to
\kb(\Bcal)$ the canonical action. Then $\hat{\beta}$ is the enveloping
action of $\hat{\alpha}^\Bcal$.
\end{corollary}

In previous sections we have decomposed a weak equivalence between
Fell bundles as strong equivalence followed by globalization of
partial actions and Morita equivalence of enveloping actions
(\ref{thm:main}, \ref{theorem every Fell bundle is strongly equivalent
  to a semidirect product bundle} and \ref{theorem interpretation of
  morita equivalence in terms of Morita equivalence of actions}).
Combining this decomposition with the previous Corollary we obtain the
following result.

\begin{corollary}
  If $\Acal$ and $\Bcal$ are weakly equivalent Fell bundles, then
  $\hat{\al}^{\Acal}$ and $\hat{\alpha}^{\Bcal}$ have the same
  enveloping action.
\end{corollary}
\begin{proof}
  Let $\mu$ and $\nu$ be the canonical actions on $\kb(\Acal)$ and $\kb(\Bcal)$, respectively.
  The Fell bundle associated to $\mu|_{\Kb(L^2(\Acal))}$ is strongly Morita equivalent to $\Acal$.  Hence the partial action on the spectrum of $\Kb(L^2(\Acal))$ induced by $\mu$ is isomorphic to $\al$ and its enveloping action is the one induced by $\mu$ on $\spec{\kb(\Acal)}$, $\spec{\mu}$. For the same reasons $\spec{\nu}$ is an enveloping action of $\be$. We also know $\mu$ and $\nu$ are Morita equivalent, so $\spec{\mu}$ is isomorphic to $\spec{\nu}$ and this implies $\spec{\mu}$ is an enveloping action of $\be$.
\end{proof}
\begin{proposition}\label{prop:sat=global}
 A Fell bundle $\Bcal$ is saturated if and only if its associated
 partial action $\hat{\al}^{\Bcal}$ is global.
\end{proposition}
\begin{proof}
 If $\Bcal$ is saturated then $D^\Bcal_t = B_t B_\tmu =B_e$ for all $t\in G$. Thus the open set of $\spec{B_e}$ corresponding to $D^\Bcal_t$, $U_t$, is $\spec{B_e}$ itself for all $t\in G$.
 In other words, $\hat{\al}^\Bcal$ is global.

 Conversely, in case $\hat{\al}^\Bcal$ is global we have
 $U_t=\spec{B_e}$ for all $t\in G$.
 Since the correspondence between \cstar{}ideals of $B_e$ and open sets of $\spec{B_e}$ is bijective, we conclude that $ B_tB_\tmu = D^\Bcal_t=B_e$ for all $t\in G$.
 Then for every $r,s\in G$, considering $B_{rs}$ as a left $B_e$-module, we deduce that $ B_{rs} = B_eB_{rs} =  B_rB_\rmu B_{rs} \subset  B_r B_s \subset B_{rs}$, so $\Bcal$ is saturated.
\end{proof}

The last proposition implies that saturation is an
invariant of strong equivalence:

\begin{corollary}\label{cor:strong-and-sat}
Let $\Xcal$ be an $\Acal-\Bcal$ strong equivalence bundle. Then
$\Acal$ is saturated if and only if $\Bcal$ is saturated.
\end{corollary}
\begin{proof}
If two partial actions are isomorphic, then one of them is global if
and only if so is the other. Therefore our claim follows from
\ref{prop:sat=global}.
\end{proof}

\subsection{Partial actions on primitive ideal spaces.}
Consider a Fell bundle $\Bcal=(B_t)_{t\in G}$, and let $\beta$ be
the canonical action of $G$ on $\kB$. Let $\Acal$ be the Fell bundle
associated to the partial action
$\alpha:=\beta|_{\mathbb{K}(L^2(\Bcal))}$. In
particular $A_e=\mathbb{K}(L^2(\Bcal))$, and
$\hat{\alpha}=\hat{\alpha}^\Acal$. By Theorem~\ref{theorem every Fell
  bundle is strongly equivalent to a semidirect product bundle} we know that $L^2(\Bcal)$ is a strong $\Acal-\Bcal$
equivalence. In particular we have the Rieffel homeomorphisms
$\mathsf{h}:\hat{B}_e\to\hat{A}_e$ and
$\tilde{\mathsf{h}}:\Prim(B_e)\to\Prim(A_e)$.
Given a \cstar{}algebra $A$, let $\kappa:\hat{A}\to\Prim(A)$ be the map
given by $\kappa([\pi])=\ker\pi$. Then the Rieffel homeomorphisms
satisfy $\tilde{\mathsf{h}}\kappa=\kappa \mathsf{h}$. According to
\cite{Abadie:Enveloping}, $\alpha$ induces a partial action
$\tilde{\alpha}$ of $G$ on $\Prim{A_e}$, which is determined by
$\tilde{\alpha}_t(\kappa([\pi]))=\kappa(\hat{\alpha}_t([\pi]))$. Here
$\tilde{\alpha}_t:\mathcal{O}_{t^{-1}}\to \mathcal{O}_t$, where
$\mathcal{O}_t:=\{P\in\Prim{A_e}:P\not\supseteq D_t^\Acal\}$. This
partial action $\tilde{\alpha}$ is continuous, because it is a
restriction of the global action $\tilde{\beta}$ induced by $\beta$ on
$\Prim(\kB)$. Now, conjugating $\tilde{\alpha}$ by $\tilde{\mathsf{h}}$, we
obtain a continuous partial action $\tilde{\alpha}^\Bcal$ of $G$ on
$\Prim(B_e)$, which satisfies
$\kappa\hat{\alpha}_t^\Bcal([\pi])=\tilde{\alpha}_t^\Bcal\kappa([\pi])$ for all
$[\pi]\in \mathcal{V}_{t^{-1}}$.
Thus we have:

\begin{theorem}\label{thm:prim}
Every Fell bundle $\Bcal$ over the locally compact Hausdorff group $G$
induces a continuous partial action
$\tilde{\alpha}^\Bcal=(\{\mathcal{O}_t\}_{t\in G},\{\tilde{\alpha}^\Bcal_t\}_{t\in
  G})$ of $G$ on $\Prim(B_e)$, which is given by
 $\tilde{\alpha}_t(P)=B_tPB_t^{*}$ for all $P\in\mathcal{O}_{t^{-1}}$;
 hence the following diagram is commutative for all $t\in G$:
 \qquad\raisebox{4ex}{$\xymatrixrowsep{4ex}
\xymatrixcolsep{8ex}\xymatrix
{\mathcal{V}_{t^{-1}}\ar[r]^{\hat{\alpha}_t^\Bcal}\ar[d]_\kappa&\mathcal{V}_t\ar[d]^\kappa\\
\mathcal{O}_{t^{-1}}\ar[r]_{\tilde{\alpha}_t^\Bcal}&\mathcal{O}_t
}
$  }

Moreover,
\begin{enumerate}
\item if $\Xcal=(X_t)$ is a strong $\Acal-\Bcal$
equivalence bundle, the
Rieffel homeomorphism
$\tilde{\mathsf{h}}_{X_e}:\Prim(B_e)\to\Prim(A_e)$ is an isomorphism
between $\tilde{\alpha}_t^\Bcal$ and $\tilde{\alpha}_t^\Acal$.
\item If $\beta$ is the canonical action of $G$ on $\kB$, then
  $\tilde{\beta}:G\times\Prim(\kB)\to\Prim(\kB)$ is the enveloping
  action of $\tilde{\alpha}^\Bcal$.
\end{enumerate}
\end{theorem}
\begin{proof}
We only need to prove (2), since the remaining statements  follow at
once from the definition of $\tilde{\alpha}^\Bcal$. Now assertion (2)
is a direct consequence of (1) and
\cite[Proposition~7.4]{Abadie:Enveloping}.
\end{proof}

The moral of the preceding section is that a Fell bundle is
essentially the same object that a semidirect product Fell bundle for
a partial action, in the sense that it is always strongly equivalent
to such a product.
With this in mind, one should be able to translate results
from semidirect product Fell bundles to arbitrary Fell bundles.

As an
example, we have the following generalization of
\cite[Corollary~7.2]{Abadie:Enveloping}:

\begin{corollary}\label{cor:prim}
Let $\Bcal$ be a Fell bundle over the locally compact Hausdorff
group~$G$. If
$\Prim(B_e)$ is compact, then there exists an open subgroup $H$ of $G$
such that the reduction of $\Bcal$ to $H$ is a saturated Fell
bundle. In particular, if $G$ is a connected group, then $\Bcal$ is a
saturated Fell bundle.
\end{corollary}
\begin{proof}
Since $\Prim(B_e)$ is compact,
\cite[Proposition~1.1]{Abadie:Enveloping} shows there
exists an open subgroup $H$ of $G$ for which the
restriction of $\tilde{\alpha}^\Bcal$ to $H$ is a global action, that
is $D_t^\Bcal=B_e$ for all $t\in H$. Thus the reduction of $\Bcal$ to
$H$ is a saturated Fell bundle. Since the only open subgroup of a
connected group is the group itself, the proof is finished.
\end{proof}
\section{\texorpdfstring{$\contz(X)$-}{C0(X)-}Fell bundles and amenability}
\label{sec:Fell-bundles-amenability}

Given a \LCH space $X$, a \cstar{}algebra $C$ is a $\contz(X)$\nb-algebra if there exists a nondegenerate \Star{}homomorphism $\phi\colon \contz(X)\to ZM(C)$.
In this situation there exists a unique continuous function $f_\phi\colon \spec{C}\to X$ such that
$$ \overline{\pi}(a) = a(f_\phi([\pi]))1_\pi\quad\mbox{for all } \ a\in \contz(X)\mbox{ and }[\pi]\in \spec{C},$$
where $\overline{\pi}$ is the natural extension of the irreducible representation $\pi\colon C\to \Bb(\Hcal)$ to $M(C)$ and $1_\pi$ is the identity operator of $\Hcal$.

Assume now that $\theta$ is an action of $G$ on $\contz(X)$, $\be$ an action of $G$ on $C$ and that $\phi$ is equivariant in the sense that, for all $t\in G$, $a\in \contz(X)$ and $c\in C:$ $\be_t(\phi(a)c)=\phi(\theta_t(a))\beta_t(c)$.
In this situation $f_\phi$ is $\spec{\be}-\spec{\theta}$\nb-equivariant and the Fell bundle $\Bcal_\be$ is a $\spec{\theta}$\nb-Fell bundle in the following sense.

\begin{definition}
 Let $\sigma$ be an action of the \LCH group $G$ on the \LCH space $X$.
 A $\sigma$\nb-Fell bundle is a Fell bundle over $G$, $\Bcal$, for which there exists a continuous function $f\colon \spec{B_e}\to X$ which is a morphism of partial actions between $\spec{\al}$ and $\sigma$.
\end{definition}

The example that motivated this definition has a converse.
Suppose $\be$ is an action of $G$ on the \cstar{}algebra $B$ and that $\Bcal_\be$ is a $\sigma$\nb-Fell bundle.
Then the unit fiber of $\Bcal$ is $B$ and the action defined by $\Bcal_\be$ on $\spec{B}$ is the action defined by $\be$, $\spec{\be}$.
By hypothesis there exists a $\spec{\be}-\sigma$\nb-equivariant
continuous function $f\colon \spec{B} \to X$.
Since the points of $X$ are closed, there exists (by \cite[Lemma
C.6]{Williams:crossed-products}) a unique continuous function $g\colon
\prim(B)\to X$ such that $g\circ \kappa=f$, where $\kappa\colon
\spec{B}\to \prim(B)$ is given by $\kappa([\pi])=\ker(\pi)$, as in
the preceding section.
The condition $g\circ \kappa=f$ ensures that $g\colon \prim(B)\to X$
is equivariant, considering on $\prim(B)$ the action induced by
$\be$.
Using Dauns-Hofmann Theorem we conclude that there exists a unique nondegenerate and equivariant \Star{}homomorphism $\phi\colon \contz(X)\to ZM(B)$, where the action considered on $\contz(X)$ is the one defined by $\sigma$.

\begin{theorem}\label{theorem equivalence sigma fell bundle}
  Let $\Bcal$ be a Fell bundle over $G$ and $\sigma$ an action of $G$ on the \LCH space $X$.
  If $\be$ is the canonical action of $G$ on $\kB$ and $\theta$ is the action on $\contz(X)$ defined by $\sigma$, then the following are equivalent:
  \begin{enumerate}
   \item $\Bcal$ is a $\sigma$\nb-Fell bundle.
   \item There exists a nondegenerate \Star{}homomorphism $\phi\colon \contz(X)\to ZM(\kB)$ such that for all $t\in G$, $a\in \contz(X)$ and $k\in \kB$, $\be_t(\phi(a)k)=\phi(\theta_t(a))\beta(k)$.
  \end{enumerate}
\end{theorem}
\begin{proof}
  By the comments preceding the statement, the implication (1)$\Rightarrow$(2) will follow after we show that $\Bcal_\be$ is a $\sigma$\nb-Fell bundle.
  Assume that $f\colon \spec{B_e}\to X$ is an equivariant continuous function.
  If we denote by $\al$ the restriction of $\be$ to
  $A:=\Kb(L^2(\Bcal))$ and $\mathsf{h}\colon \spec{B_e}\to \spec{A}$ is the
  Rieffel homeomorphism given by the equivalence bimodule $L^2(\Bcal)$, then $f\circ \mathsf{h}^{-1} \colon \spec{A}\to X$ is equivariant.
  By \cite[Proposition 7.4.]{Abadie:Enveloping} $\spec{\be}$ is the enveloping action of $\spec{\al}$ and by \cite[Theorem 1.1.]{Abadie:Enveloping} there exists a unique $\spec{\be}-\sigma$\nb-equivariant continuous extension of $f\circ \mathsf{h}^{-1}$.
\end{proof}

The next result is an extension of \cite[Theorem~5.3.]{Anantharaman-Delaroche:Amenability} to Fell bundles.

\begin{theorem}
 Let $\sigma$ be an action of $G$ on the \LCH space $X$.
 Consider the conditions:
 \begin{enumerate}
  \item $\sigma$ is amenable.
  \item Every $\sigma$\nb-Fell bundle is amenable, that is, $C^*(\Bcal)=C^*_r(\Bcal)$.
  \item For every $\sigma$\nb-Fell bundle $\Bcal$ with $B_e$ nuclear, $C^*_r(\Bcal)$ is nuclear.
  \item $\contz(X)\rtimes_r G$ is nuclear.
 \end{enumerate}
 Then $(1)\Rightarrow(2)\Rightarrow(3)\Rightarrow(4)$ and
 if $G$ is discrete $(4)\Rightarrow(1)$.
\end{theorem}
\begin{proof}
  Name $\be$ the canonical action on $\kB$.
  Since $\Bcal$ is equivalent to $\Bcal_\be$, $\Bcal$ is amenable if and only if $\Bcal_\be$ is amenable.
  Moreover, since $C^*_r(\Bcal)$ and $C^*_r(\Bcal_\be)$ are Morita equivalent, one is nuclear if and only if the other one is.

  Assume (1) holds.
  By \cite[Theorem 5.3.]{Anantharaman-Delaroche:Amenability} and Theorem \ref{theorem equivalence sigma fell bundle}, $\Bcal_\be$ is amenable and so $\Bcal$ is amenable.
  Now assume that (2) holds, then (2) from \cite[Theorem 5.3.]{Anantharaman-Delaroche:Amenability} holds and it suffices to show that $\kB$ is nuclear.
  We know $\Kb(L^2(\Bcal))$ is nuclear because $B_e$ is nuclear.
  Then \cite[Proposition 2.2.]{Abadie:Enveloping} implies that $\kB$ is nuclear.

  The rest of the proof follows directly from \cite[Example~(3) of
  4.4. together with
  Theorem~5.8]{Anantharaman-Delaroche:Amenability}.
\end{proof}

\appendix

\section{Tensor products of equivalence bundles}\label{appendix:on tensor products}

Throughout this section we use the construction of adjoint and tensor product of equivalence bundles of \cite{Abadie-Ferraro:equivFB}.

\begin{theorem}\label{theorem:strong equivalence is an equivalence relation}
 If $\Xcal$ and $\Ycal$ are $\Acal-\Bcal$ and $\Bcal-\Ccal$-strong equivalence bundles, respectively, then the tensor product bundle $\Zcal:=\Xcal\otimes_\Bcal \Ycal$ is (left and right) strongly full.
 In particular, strong equivalence of Fell bundles is an equivalence relation.
\end{theorem}
\begin{proof}
  To show that $\Zcal$ is (right) full the authors show that given $r,s\in G$, $x_1,x_2\in X_r$, $y_1,y_2\in Y_s$ and $\vep>0$ there exists $\xi_1,\xi_2\in Z_{rs}$ such that $\| \lac y_1,\lab x_1,x_2\rab y_2\rac - \lac \xi_1,\xi_2\rac \|<\vep$. Now we use that fact to show $\Zcal$ is strongly full.

  Fix $t\in G$, $c\in C_t^*C_t$ and $\vep >0$.   Since $\Ycal$ is strongly full there exists $y_{j,k}\in Y_t$ ($j=1,2$ and $k=1,\ldots,n$) such that $\| c - \sum_{k=1}^n \lac y_{1,k},y_{2,k}\rac \|<\vep$.   We also know that $\Xcal$ is strongly full, then $X_e$ is a $A_e-B_e$-equivalence bimodule and we can find $x_{j,k}\in X_e$ ($j=1,2$ and $k=1,\ldots,m$) such that
  $$ \delta:=\left\| c - \sum_{k=1}^n\sum_{l=1}^m \lac y_{1,k}, \lab x_{1,l},x_{2,l}\rab y_{2,k}\rac \right\|<\vep.$$
  From the first paragraph of this proof it follows that we may find $\xi_{p,k,l}\in Z_t$ ($p=1,2$, $k=1,\ldots,n$, and $l=1,\ldots,m$) such that $\| \lac y_{1,k}, \lab x_{1,l},x_{2,l}\rab y_{2,k}\rac - \lac \xi_{1,k,l}, \xi_{2,k,l}\rac \|< \frac{\vep-\delta}{nm}. $
  Thus
  $$\left\| c - \sum_{k=1}^n\sum_{l=1}^m \lac \xi_{k,l,1}, \xi_{k,l,2}\rac \right\|<\vep. $$

  By symmetry, $\Zcal$ is also left strongly full.
  Hence strong equivalence is a transitive relation.
  Regarding the symmetric and reflexive properties of strong equivalence, we leave to the reader the verification of the fact that the adjoint of $\Xcal$, $\tilde{\Xcal}$, is a full $\Bcal - \Acal$-equivalence bundle and that $\Bcal$, considered as a $\Bcal-\Bcal$-equivalence bundle in the natural way, are left and right strongly full.
\end{proof}

An equivalence module ${}_AX_B$ between the \cstar algebras $A$ and $B$ is usually viewed as an arrow from $A$ to $B$; here we view it as an arrow from $B$ to $A$ to be consistent with our convention at the beginning of Section~\ref{sec:fb and glob} where we view Fell bundles as actions by equivalences. The composition of arrows is given by inner tensor product. Although the tensor products $(X\otimes_B Y)\otimes_C Z$ and $X\otimes_B (Y\otimes_C Z)$ are not the same object, but they are naturally isomorphic (via a unitary); this gives the associativity of composition. Then we obtain a category with \cstar{}algebras as objects and unitary equivalence classes of equivalence modules as objects. 
To proceed analogously with Fell bundles we need a notion of unitary operator between equivalence bundles.

\begin{definition}\label{def:unitaries}
  Let $\Xcal$ and $\Ycal$ be two $\Acal-\Bcal$-equivalence bundles. A \emph{unitary} from $\Xcal$ to $\Ycal$ is an isomorphism of equivalence bundles $\rho\colon \Xcal\to \Ycal$ such that $\laa \rho(x),\rho(y)\raa = \laa x,y\raa$ and $\lab \rho(x),\rho(y)\rab = \lab x,y\rab$ (this means that $\rho^l = \id_\Acal$ and $\rho^r=\id_\Bcal$ in the notation of \cite{Abadie-Ferraro:equivFB}). If such an isomorphism exists, we say that $\Xcal $ is \emph{unitarily equivalent} (or just \emph{isomorphic}) to $\Ycal$.
\end{definition}

To obtain a category with Fell bundles (over a fixed group $G$) as objects, isomorphism classes of equivalence bundles as morphisms and the tensor product of \cite{Abadie-Ferraro:equivFB} as composition we need to show that the composition is well defined on the of isomorphism classes and that it is associative. This boils down to the following results.

\begin{proposition}\label{prop:composition of morphism is defined}
  Suppose $\pi\colon \Xcal_1\to \Xcal_1$ is a unitary between $\Acal-\Bcal$-equivalence bundles and $\rho\colon \Ycal_1\to \Ycal_2$ is a unitary between $\Bcal-\Ccal$-equivalence bundles.
  Then $\Xcal_1\otimes_\Bcal\Ycal_1$ is unitarily equivalent to $\Xcal_2\otimes_\Bcal\Ycal_2$.
\end{proposition}
\begin{proof}
  The way the tensor product is constructed is one of the key factors of this proof, so it will be necessary to recall it here.
  Start by considering the bundle $\Zcal_j:=\{{X_j}_r\otimes_{B_e}{Y_j}_s\}_{(r,s)\in G\times G}$, $j=1,2$.   The topology of that bundle is determined by the set of sections $\Gamma_j:=\spn \{f\boxtimes g\colon f\in \contc(\Xcal_j),\ g\in \contc(\Ycal_j)\}$ where $f\boxtimes g(r,s)=f(r)\otimes g(s)$.   Note \cite[II 13.16]{Doran-Fell:Representations} implies the existence of a unique isomorphism of Banach bundles $\mu\colon \Zcal_1\to \Zcal_2$ such that $\mu(x\otimes y)=\pi(x)\otimes \rho(y)$.
  Recall that the construction of $\Xcal_j\otimes_\Bcal\Ycal_j$ is performed using actions of $\Acal$ and $\Ccal$ on $\Zcal_j$  and operations $\triangleleft_j\colon \Zcal_j\times \Zcal_j\to \Ccal$ and $\triangleright_j\colon \Zcal_j\times \Zcal_j\to \Acal$ uniquely determined by the identities
  \begin{align*}
    a(x\otimes y) & = (ax)\otimes y   &  (x\otimes y)c &= x\otimes (yc)\\
    (x\otimes y)\triangleleft_j(u\otimes v) & = \laa x {}_\Bcal\la y,v\ra ,u\raa & (x\otimes y)\triangleright_j(u\otimes v) &= \lac y, \lab x,u\rab u\rac
  \end{align*}
  The reader can easily check that $a\mu(z)=\mu(az)$, $\mu(z)c=\mu(z)$, $\mu(z)\triangleleft_2 \mu(z') = z \triangleleft_1 z'$ and $\mu(z)\triangleright_2 \mu(z') = z \triangleright_2 z'$ for all $z,z'\in \Zcal_1$, $a\in \Acal$ and $c\in \Ccal$.

  If we think of $\Zcal_1$ and $\Zcal_2$ as the same object, then there is nothing else to prove and $\Xcal_1\otimes_\Bcal\Ycal_1$ is in fact the same as $\Xcal_2\otimes_\Bcal\Ycal_2$.
  In other case the next step is to define, for every $t\in G$, ${U_j}_t$ as the reduction of $\Zcal_j$ to $H^t:=\{(r,s)\in G\times G\colon rs=t\}$.
  Then we get an untopologized bundle $\Ucal_j:=\{ {U_j}_t \}_{t\in G}$ and define pre-inner product and actions in the following way:
  \begin{align}
    \plauja{j} u,v \prauja & :=\iint_{G\times G} u(p,\pmu r)\triangleleft_j
    v(q,\qmu s)\dd p\dd q; \\
    \plaujc u,v \praujc{j} &:=\iint_{G\times G} u(p,\pmu r)\triangleright_j
    v(q,\qmu s)\dd p\dd q; \\
    au\in {U_j}_{tr} &\mbox{ by the formula } (au)(p,\pmu tr):= au(\tmu p,\pmu
    tr)\mbox{ and }\\
    uc\in {U_j}_{rt} & \mbox{ by the formula } (uc)(p,\pmu rt):= u(p,\pmu
    r)c,
  \end{align}
  where $u\in {U_j}_r$, $v\in {U_j}_s$, $a\in A_t$ and $c\in C_t$.

  It is then clear that the composition with $\mu$ identifies the pre-inner products an actions of $\Ucal_1$ and $\Ucal_2$, for example $\plauja{1} u,v \prauja = \plauja{2} \mu\circ u,\mu\circ v \prauja$.   Each fiber ${U_j}_t$ is a seminormed space when considered with the seminorm $\|u\|:= \|\plauja{j} u,u \prauja  \|^{1/2} = \|\plaujc u,u \praujc{j}  \|^{1/2}$.   The space $[{U_j}_t]$ is defined as the quotient of ${U_j}_t$ by the subspace of zero length vectors, where square brackets are used to represent equivalence classes.

  The tensor product $\Xcal_1\otimes_\Bcal \Ycal_1$ is obtained by completing each fiber of $[\Ucal_j]=\{ [{U_j}_t]\}_{t\in G}$ and a set of continuous sections of this tensor product is given by those of the form $[\xi]$, for $\xi\in \contc(\Zcal_j)$, where $[\xi](t)=[\xi|_{H^t}]$ and $\xi|_{H^t}$ represents the restriction of $\xi$ to $H^t$.

  Note there exists a unique bijective isometry $\mu^*_t\colon [{U_1}_t]\to [{U_2}_t]$ such that $[u]\mapsto [\mu\circ u]$. Then there exists a unique function $\mu^*\colon \Xcal_1\otimes_\Bcal \Ycal_1\to \Xcal_2\otimes_\Bcal \Ycal_2$ which is linear and bounded on each fiber and extends each $\mu^*_t$. Clearly, $\mu^*$ is an isometry and $\mu^*\circ [\xi] = [\mu\circ \xi]$ for all $\xi\in \contc(\Zcal_1)$. In this situation \cite[II 13.16]{Doran-Fell:Representations} implies $\mu^*$ is an isomorphism of Banach bundles.
  Moreover, it preserves the left and right inner products because $\mu$ transforms the inner products and actions of the bundle $\{{U_1}_t\}_{t\in G}$ to those of $\{{U_2}_t\}_{t\in G}$.
\end{proof}

\begin{proposition}\label{prop:identity morphism,invertible arrows, associativity}
  Let $\Xcal,\ \Ycal$ and $\Zcal$ be $\Acal-\Bcal$, $\Bcal-\Ccal$ and $\Ccal-\Dcal$-equivalence bundles, respectively.
  Then
  \begin{enumerate}[(a)]
   \item\label{item:identity morphisms} $\Acal\otimes_\Acal \Xcal$ is unitarily equivalent to $\Xcal$ and $\Xcal\otimes_\Bcal \Bcal$ to $\Bcal$.
   \item\label{item:every arrow is invertible} $\widetilde{\Xcal}\otimes_{\Acal}\Xcal$ is unitarily equivalent to $\Bcal$ and $\Xcal\otimes_\Bcal \widetilde{\Xcal}$ to $\Acal$.
   \item\label{item:tensor product associative} The tensor products $(\Xcal\otimes_\Bcal \Ycal)\otimes_\Ccal \Zcal$ and $\Xcal\otimes_\Bcal (\Ycal\otimes_\Ccal \Zcal)$ are unitarily equivalent.
  \end{enumerate}
\end{proposition}
\begin{proof}
  The proofs of the two claims in (\ref{item:identity morphisms}) are analogous, thus we just prove the first one; the same comment holds for (\ref{item:every arrow is invertible}).

  Let $\Zcal$ be the bundle constructed from $\Acal$ and $\Xcal$ ($Z_{(r,s)}=A_r\otimes_{A_e}X_s$) as in the proof of \ref{prop:composition of morphism is defined}.
  We claim that there exists a unique continuous map $\pi\colon \Zcal\to \Xcal$ such that: $\pi( Z_{(r,s)} )\subset X_{rs}$, $\pi|_{Z_{(r,s)}}$ is linear and $\pi(a\otimes x)=ax$ for all $r,s\in G$, $a\in \Acal$ and $x\in \Xcal$. First note there exists a unique linear isometry $\pi_{r,s}\colon Z_{(r,s)}\to X_{rs}$ sending $a\otimes x$ to $ax$ because, for every $a_1,\ldots,a_n\in A_r$ and $x_1,\ldots,x_n$ we have
  $$ \lab \sum_{i=1}^n a_i\otimes x_i,\sum_{i=1}^n a_i\otimes x_i\rab =\sum_{i,j=1}^n \lab x_i,a_i^*a_j x_j\rab  = \lab \sum_{i=1}^n a_ix_i,\sum_{i=1}^n a_ix_i\rab . $$
  If $\pi\colon \Zcal\to \Acal$ is the unique map extending all the $\pi_{r,s}$, then, using \cite[II 13.16]{Doran-Fell:Representations}, we conclude that $\pi$ is continuous because for all $f\in \contc(\Acal)$ and $g\in \contc(\Xcal)$, $\pi\circ (f\boxtimes g)$ is continuous.

  Now let $\Ucal$ be constructed from $\Zcal$ as in the proof of \ref{prop:composition of morphism is defined}.
  For $f,u\in \contc(\Acal)$ and $g,v\in \contc(\Xcal)$ we have
  $$ \la f\boxtimes g|_{H^r} , u\boxtimes v|_{H^s}\ra^\Ucal_\Bcal = \la \int_G \pi(  f\boxtimes g (p,\pmu r))\dd p,\int_G \pi(  u\boxtimes v (q,\qmu s))\dd q \ra_\Bcal.$$
  Since $\spn\{ f\boxtimes g|_{H^t}\colon  f\in \contc(\Acal),\ g\in \contc(\Xcal)\}$ is dense in the inductive limit topology of $U_t$ and $\la\ , \ \ra^\Ucal_\Bcal$ is continuous on each variable (separately) with respect to this topology, we conclude that
  \begin{equation}\label{equ:rigth tensor product A otimes X}
    \la \xi , \eta \ra^\Ucal_\Bcal = \la \int_G \pi( \xi (p,\pmu r))\dd p,\int_G \pi(  \eta (q,\qmu s))\dd q \ra_\Bcal
  \end{equation}
  for all $\xi\in U_r$, $\eta\in U_s$ and $r,s\in G$.

  Then we can define a map $\mu\colon [\Ucal] \to \Xcal$ such that, for $\xi\in U_r$,
  $$\mu([\xi]) = \int_G\pi( \xi (p,\pmu r))\dd p .$$
  This map is linear and isometric on each fiber, so it can be (continuously) extended to the closure of each fiber.
  The resulting extension is a map $\mu \colon \Acal\otimes_\Acal\Xcal\to \Xcal$.
  Recall from \cite{Abadie-Ferraro:equivFB} that the topology of $\Acal\otimes_\Acal\Xcal$ is constructed using the sections of the form $r\mapsto [\eta|_{H^r}]$, where $\eta\in \contc(\Zcal)$. Besides, for every $\eta\in \contc(\Zcal)$, the section $r\mapsto \mu([\eta|_{H^r}])$ is continuous (this is not immediate but can be proved by standard arguments, see for example the ideas developed in \cite[II 15.19]{Doran-Fell:Representations}).
  Then \cite[II 13.16]{Doran-Fell:Representations} implies $\mu$ is continuous.
  Note Equation \ref{equ:rigth tensor product A otimes X} implies $\mu^r=\id_\Bcal$ and the reader can show that $\mu^l=\id_\Acal$ with analogous computations.
  Then all we need to do is to show $\mu$ is surjective or, alternatively, that $\mu([U_r])$ is dense in $X_r$ for all $r\in G$.

  Fix $x\in X_r$ and take $a\in A_e$ and $x'\in X_r$ such that $ax'=x$. Now take $f\in \contc(\Acal)$ and $g\in \contc(\Xcal)$ such that $f(e)=a$ and $g(r)=x'$. Denote $I$ the directed set of compact neighbourhoods of $e\in G$ with respect the usual order: $i\leq j$ if $j\subset i$. For each $i\in I$, take a function $\varphi_i\in \contc(G)^+$ with $\int_G \varphi_i(t)\dd t=1$ and $\supp(\varphi_i)\subset i$. Then $\lim_i \mu((\varphi_i f)\boxtimes g|_{H^r})=f(e)g(r) = x$; we conclude that $\mu$ is surjective.
  This implies that $\mu$ is unitary.

  The proof of (\ref{item:every arrow is invertible}) is very similar to that of (\ref{item:identity morphisms}).
  We start by constructing a surjective isometry $\pi\colon \Zcal\to \Bcal$, where $\Zcal=\{ \widetilde{X_\rmu}\otimes_{A_e}X_s \}_{(r,s)\in G\times G}$. Take $r,s\in G$, $x^j_1,\ldots,x^j_n\in X_\rmu$ and $y^j_1,\ldots,y^j_n\in X_s$ ($j=1,2$).
  Then note that
  \begin{align*}
    \sum_{i=1}^n \widetilde{x^1_i}\otimes y^1_i \triangleright \sum_{j=1}^n \widetilde{x^2_j}\otimes y^2_j
      & = \sum_{i,j=1}^n  \lab y^1_j,\laa x^1_i,x^2_j\raa y^2_j\rab
        = \sum_{i,j=1}^n  \lab y^1_j,x^1_i \lab x^2_j, y^2_j\rab\rab  \\
      & = \left( \sum_{i=1}^n \lab x^1_i,y^1_i\rab \right)^* \sum_{i=1}^n \lab x^2_i,y^2_i\rab.
  \end{align*}
  Besides, the restriction of $\triangleright$ to $Z_{(r,s)}\times Z_{(r,s)}$ is the inner product of $Z_{(r,s)}$. Then we conclude there exists a unique map $\pi\colon \Zcal\to \Bcal$ such that: $\pi(\widetilde{x}\otimes y)=\lab x,y\rab$, $\pi(Z_{(r,s)})\subset \Bcal_{rs}$ and $\pi|_{Z_{(r,s)}}$ is linear for all $x,y\in \Xcal$ and $r,s\in G$.
  Moreover, $z\triangleright w = \pi(z)^*\pi(w)$, $\pi(z b)=\pi(zb)$ and $\pi(bz)=b\pi(z)$ for all $z,w\in \Zcal$ and $b\in \Bcal$. Note also that $\pi$ is continuous because, for $f,g\in \contc(\Xcal)$, we have $\pi\circ \widetilde{f}\boxtimes g(r,s) = \lab f(\rmu),g(s)\rab$ and so $\pi\circ \widetilde{f}\boxtimes g$ is continuous.

  To complete the proof of (\ref{item:every arrow is invertible}) it suffices to follow the steps of the proof of (\ref{item:identity morphisms}), using the map $\pi$ we have just constructed instead of the map $\pi$ we used to prove (\ref{item:identity morphisms}).

  We now deal with (\ref{item:tensor product associative}).
  Let $[\Ucal]$ and $[\Vcal]$ be the bundles whose fiber-wise completion gives $\Xcal\otimes_\Bcal\Ycal$ and $\Ycal\otimes_\Ccal\Zcal$, respectively (see the proof of Proposition \ref{prop:composition of morphism is defined}).
  For every pair $(f,g)\in \contc(\Xcal)\times \contc(\Ycal)$ we have a section $[f\boxtimes g]\in \contc(\Xcal\otimes_\Bcal\Ycal)$ such that $[f\boxtimes g](t) = [f\boxtimes g|_{H^t}]\in [U_t]$. Define $\Gamma_\Ucal$ as the linear span of the sections $[f\boxtimes g]$. Part of the construction of $\Xcal\otimes_\Bcal\Ycal$ is based on the fact that $\{\xi(t)\colon \xi\in \Gamma_\Ucal\}$ is dense in $[U_t]$ and so in the fiber over $t$ of $\Xcal\otimes_\Bcal\Ycal$. Of course that the same holds for $\Gamma_\Vcal$.

  Fix $r,s \in G$, $f,u\in \contc(\Xcal)$, $g,v\in \contc(\Ycal)$ and $h,w\in \contc(\Zcal)$. We want to prove that
  \begin{equation}\label{equ:inner product identity for double tensor products}
   \la [[f  \boxtimes g]   \boxtimes h](r) , [[u\boxtimes v] \boxtimes w](s)\ra_\Dcal = \la [f\boxtimes [g\boxtimes h]](r) , [u\boxtimes [v\otimes w]](s)\ra_\Dcal.
  \end{equation}
  To do this first note that, using the definitions of the inner product of tensor products bundles, we obtain
  \begin{align*}
   \la [[f & \boxtimes g]   \boxtimes h](r) , [[u\boxtimes v] \boxtimes w](s) \ra_\Dcal  \\
    & = \int_{G^2} [f\boxtimes g] \boxtimes h (p,\pmu r)\triangleright [u\boxtimes v] \boxtimes w (q,\qmu s)\dd p\dd q \\
    & = \int_{G^2} \la h (\pmu r), \la  [f\boxtimes g|_{H^p}], [u\boxtimes v|_{H^q}]\ra_{\Ccal} w (\qmu s)\ra_\Dcal\dd p\dd q \\
    & = \int_{G^4} \la h (\pmu r), \la g(x^{-1}p) , \la f(x),u(y)\ra_\Bcal v(y^{-1}q)\ra_{\Ccal} w (\qmu s)\ra_\Dcal\dd x\dd y\dd p\dd q.
  \end{align*}
  Using the substitutions $p\mapsto xp$ and $q\mapsto xq$ in the integrals, we obtain
  \begin{align*}
   \la [[f & \boxtimes g]   \boxtimes h](r) , [[u\boxtimes v] \boxtimes w](s)\ra_\Dcal  \\
    & = \int_{G^4} \la h (\pmu x^{-1} r), \la g(p) , \la f(x),u(y)\ra_\Bcal v(y^{-1}xq)\ra_{\Ccal} w (\qmu x^{-1}s)\ra_\Dcal\dd x\dd y\dd p\dd q \\
    & = \int_{G^4}  (g(p)\otimes h(\pmu x^{-1} r)) \triangleright ( \la f(x),u(y)\ra_\Bcal v(y^{-1}xq)\otimes  w (\qmu x^{-1}s) )  \dd x\dd y\dd p\dd q \\
    & = \int_{G^2} \la [g\boxtimes h](x^{-1}r) , \la f(x),u(y)\ra_\Bcal [v\otimes w](y^{-1}s)\ra_\Dcal  \dd x\dd y\\
    & = \la [f\boxtimes [g\boxtimes h]](r) , [u\boxtimes [v\otimes w]](s)\ra_\Dcal .
  \end{align*}

  Using equation \ref{equ:inner product identity for double tensor products} and \cite[II 13.16]{Doran-Fell:Representations}
  we can justify the existence of a unique isometric isomorphism of Banach bundles $\mu\colon (\Xcal\otimes_\Bcal \Ycal)\otimes_\Ccal \Zcal\to \Xcal\otimes_\Bcal (\Ycal\otimes_\Ccal \Zcal)$ such that $\mu( [[f \boxtimes g]   \boxtimes h](r) ) = [f\boxtimes [g\boxtimes h]](r)$. We leave to the reader the verification of the fact that $\mu( [[f  \boxtimes g]   \boxtimes h](r) d ) = \mu([[f  \boxtimes g]   \boxtimes h](r))d$. After this it is immediate that $\mu(\xi d)=\mu(\xi) d$ for all $\xi\in (\Xcal\otimes_\Bcal \Ycal)\otimes_\Ccal \Zcal$ and $d\in \Dcal$. Note equation \ref{equ:inner product identity for double tensor products} implies $\la \mu(\xi),\mu(\eta)\ra_\Dcal = \la \xi,\eta\ra_\Dcal$ for all $\xi,\eta\in (\Xcal\otimes_\Bcal \Ycal)\otimes_\Ccal \Zcal$. Then $\mu$ is a morphism of equivalence bundles because
  $$ \mu(\xi\la \eta,\zeta\ra_\Dcal) = \mu(\xi) \la \eta,\zeta\ra_\Dcal = \mu(\xi) \la \mu(\eta), \mu(\zeta)\ra_\Dcal.$$
  The identity $\la \mu(\xi),\mu(\eta)\ra_\Dcal = \la \xi,\eta\ra_\Dcal$ tells us that $\mu^r=\id_\Dcal$, and we leave to the reader to check that $\mu^l=\id_\Acal$ (the proof of which is analogous to the proof of \ref{equ:inner product identity for double tensor products}).
\end{proof}

Using the last two propositions one can construct a category $\mathscr{E}_G^w$ (resp. $\mathscr{E}_G^s$) of Fell bundles over $G$ as objects and isomorphism classes of weak (resp. strong) equivalence bundles as arrows.
The identity morphism associated to the Fell bundle $\Acal$ is the isomorphism class of $\Acal$, $[\Acal]$. The composition of the arrows $\Acal\stackrel{[\Xcal]}{\to}\Bcal$ and $\Bcal\stackrel{[\Ycal]}{\to}\Ccal$ is $\Acal\stackrel{[\Xcal\otimes_\Bcal\Ycal]}{\to}\Ccal$. Propositions \ref{prop:composition of morphism is defined} and \ref{prop:identity morphism,invertible arrows, associativity} tell us that we indeed obtain a category with these definitions and, moreover, every arrow is invertible in this category. Only a weak $2$-category (or bicategory) can be formed if we do not take isomorphism class, but just the equivalence bundles as arrows. The unitaries introduced in Definition~\ref{def:unitaries} can be used as $2$-arrows for this weak $2$-category. This works similarly to the $2$-category of \cstar{}algebras with correspondences as arrows introduced in \cite{Buss-Meyer-Zhu:Higher_twisted}.

\vskip 0,5pc

\end{document}